\documentclass[11pt, a4paper, reqno]{amsart}
\usepackage{amsmath,amsfonts, amsthm,amssymb,amsfonts,amscd, graphicx, hyperref, enumerate,mathtools, mathrsfs, dsfont, upgreek, appendix, xfrac, enumitem}
\newtheorem{theorem}{Theorem}[section]

\newtheorem{lemma}[theorem]{Lemma}
\newtheorem{definition}[theorem]{Definition}
\newtheorem{prop}[theorem]{Proposition}
\theoremstyle{remark}
\newtheorem{remark}[theorem]{Remark}
\usepackage[left=2.5cm, right=2.5 cm]{geometry}

\newcommand{\eps}{\varepsilon}
\makeatletter

\hypersetup{
  colorlinks   = true, 
  urlcolor     = black, 
  linkcolor    = black, 
  citecolor    = black 
}

\mathtoolsset{showonlyrefs=true}

\newcommand{\norm}[1]{\lVert#1\rVert} 
\DeclareMathOperator*{\esssup}{ess.\,sup} 
\DeclareMathOperator*{\supp}{supp} 

\newcommand*{\R}{\ensuremath{\mathbb{R}}}

\newcommand*{\N}{\ensuremath{\mathbb{N}}}

\newcommand*{\diam}{\ensuremath{\mathrm{diam\,}}}

\def\Xint#1{\mathchoice
{\XXint\displaystyle\textstyle{#1}}%
{\XXint\textstyle\scriptstyle{#1}}%
{\XXint\scriptstyle\scriptscriptstyle{#1}}%
{\XXint\scriptscriptstyle\scriptscriptstyle{#1}}%
\!\int}
\def\XXint#1#2#3{{\setbox0=\hbox{$#1{#2#3}{\int}$ }
\vcenter{\hbox{$#2#3$ }}\kern-.6\wd0}}

\def\dashint{\Xint-}

\newcommand{\mean}[1]{\,-\hskip-1.08em\int_{#1}}
\renewcommand{\mean}[1]{\dashint_{#1}}

\AtBeginDocument{
\let\div\undefined\DeclareMathOperator{\div}{div} 
}
\AtBeginDocument{
}
\makeindex

\author[Colombo]{Maria Colombo
}
\address{EPFL SB, Station 8, 
CH-1015 Lausanne, Switzerland
}
\email{maria.colombo@epfl.ch}

\author[Haffter]{Silja Haffter
}
\address{EPFL SB, Station 8, 
CH-1015 Lausanne, Switzerland
}
\email{silja.haffter@epfl.ch}

\title[Improved Hausdorff dimension estimate of the singular set of the SQG equation]{Improved Hausdorff dimension estimate of the singular set of the supercritical surface quasigeostrophic equation}

\begin{document}
\begin{abstract}  We prove that the spacetime singular set of any suitable Leray-Hopf solution of the surface quasigeostrophic equation with fractional dissipation of order $0< \alpha < \frac{1}{2}$ has Hausdorff dimension at most $\frac{1}{2\alpha^2}\,.$   This result improves previously known dimension estimate established in \cite{CH2021} and builds on the excess decay result and the control on the particle flow already developed there.
The improvement lies in the initial iteration of the local energy inequality 
in analogy with the celebrated result of Caffarelli-Kohn-Nirenberg \cite{CaffarelliKohnNirenberg1982} for the Navier-Stokes equations.
\end{abstract}

\maketitle

\section{Introduction}

For $\alpha \in (0, \frac{1}{2}]$, we consider the following fractional drift-diffusion equation
\begin{equation}\label{eq:SQG}
\begin{cases}
\partial_t \theta + u \cdot \nabla \theta = - (-\Delta)^\alpha \theta \\
\div u =  0 \, ,
\end{cases}
\end{equation}
where $\theta: \R^2 \times [0, +\infty) \rightarrow \R$ is an active scalar, $u: \R^2 \times [0, +\infty) \rightarrow \R^2$ is the velocity field and $(-\Delta)^\alpha$ corresponds to the Fourier multiplier with symbol $\lvert \xi \rvert^{2\alpha}$. The system is usually complemented with the initial condition 
\begin{equation}\label{eq:Cauchy}
\theta (\cdot, 0) = \theta_0\, .
\end{equation}
We will be particularly interested in the surface quasigeostrophic (SQG) equation where the velocity field $u$ is determined from $\theta$ by the Riesz-transforms $\mathcal{R}$ on $\R^2$
\begin{equation}\label{eq:u}
 u = \nabla^\perp (-\Delta)^{-\frac{1}{2}} \theta = \mathcal{R}^\perp \theta\,.
\end{equation}
The smoothness and long time behaviour of solutions of  \eqref{eq:SQG}--\eqref{eq:u} is known for $\alpha \geq \frac{1}{2}$ due to the series of works \cite{KiselevNazarovVolberg2007,CaffarelliVasseur2010,ConstantinVicol2012,ConstantinTarfuleaVicol2015,ConstantinCotiZelatiVicol2016}. This range corresponds to the $L^\infty$-subcritical and -critical regime with respect to the natural scaling invariance of the equation given by
\begin{equation}\label{eq:scaling}
(\theta, u) \mapsto \big(\theta_r(x,t):=  r^{2\alpha-1} \theta(rx, r^{2\alpha}t), u_r(x,t):= r^{2\alpha-1} u(rx, r^{2\alpha} t) \big) \,.
\end{equation}
In the supercritical range $\alpha < \frac{1}{2}\,,$ it is still open whether or not starting from smooth data remain smooth forever. Partial results in a positive directions include the conditional regularity results of \cite{ConstantinWu2008,ConstantinWu2009} assuming a certain (sub-)critical spatial H\"older regularity of the velocity field as well as the partial regularity result of \cite{CH2021}, which established the almost everywhere smoothness of suitable Leray-Hopf weak solutions as a consequence of a non-trivial dimension estimate on the spacetime singular set. On the negative side, the non-uniqueness of a class of (very) weak solutions for the system \eqref{eq:SQG}--\eqref{eq:u}, even for subcritical dissipations, has been demonstrated in \cite{BuckmasterShkollerVicol2019}.

The goal of this note is to give a an estimate of the dimension of the spacetime singular set ${\rm Sing }\, \theta$ of a suitable Leray-Hopf solution (see Definition \ref{def:LH} below), which consists of all points in spacetime such that $\theta$ is not smooth in any neighbourhood.

\begin{theorem}\label{thm:main} Let $\alpha \in (\frac{1}{\sqrt 6}, \frac 12)$ and $(\theta, u)$ be a suitable Leray-Hopf solution  of \eqref{eq:SQG}--\eqref{eq:u} on $\R^2 \times (0, + \infty)\,.$ Then the Hausdorff-dimension of the spacetime singular set ${\rm Sing} \, \theta$ satisfies 
$$\dim_\mathcal{H}({\rm Sing} \, \theta) \leq \frac{1}{2\alpha^2}\,.$$
\end{theorem}

Theorem~\ref{thm:main} is an improvement of both the previously known bound  $\frac{1}{2\alpha}(\frac{1+\alpha}{\alpha}(1-2\alpha) +2 )$ of \cite{CH2021} on the dimension of the singular set  and of the range of $\alpha$ for which such bound is nontrivial. Indeed, from Theorem \ref{thm:main} it follows that for fractional orders above $\alpha_0= 1/\sqrt{6}\,,$  suitable Leray-Hopf solutions of \eqref{eq:SQG}--\eqref{eq:u}  are smooth almost everywhere on $\R^2 \times (0, + \infty)\,,$ whereas previously this was only known to hold true above $\alpha_0= (1+\sqrt{33})/16$. 


\medskip

As typical in partial regularity results, the proof of Theorem \ref{thm:main} is, with a covering argument, reduced to an $\varepsilon$-regularity result (see Theorem \ref{thm:running} below) and as a consequence, its dimension estimate is intimately linked to the scaling of the \emph{integral} quantity present in the smallness requirement of such $\varepsilon$-regularity result. The improvement with respect to the dimension estimate of \cite{CH2021} relies therefore on proving a $\varepsilon$-regularity theorem which involves only better scaling \emph{integral} quantities;  in case of the SQG equation, there is a vast choice of such \emph{globally controlled}  quantities thanks to the maximum principle and the global energy inequality. 

Our improved $\varepsilon$-Theorem \ref{thm:running} below is much inspired by the celebrated Theorem of Caffarelli-Kohn-Nirenberg for the classical Navier-Stokes equations \cite{CaffarelliKohnNirenberg1982} and  the hyperdissipative Navier-Stokes equation in \cite{CDLM17}: it is based on the simple observation that by interpolation, the smallness of the local part of the excess of \cite{CH2021} can be reached by requiring only the smallness of the kinetic and the dissipative part of a certain nonlinear local energy, both of which have a better scaling than the interpolated spacetime norm employed in \cite{CH2021}. The kinetic part though is not of integral type and cannot be used in a covering argument, whereas the non-local part of the excess, generated both by the tails of the fractional Laplacian and the Riesz-transform, has a worse scaling. To circumvent both problems, we make a crucial, iterative use of the local energy inequality à la Caffarelli-Kohn-Nirenberg which guarantees the asymptotic smallness of both the kinetic part of the nonlinear local energy (playing the same role as the local kinetic energy in Navier-Stokes) as well as the tails of $\theta$ (corresponding to the non-local part of the pressure in Navier-Stokes). In light of the $L^\infty$-supercritical nature of the SQG equation, this procedure encounters the additional difficulty that one needs to require that the velocity field has zero space average on $B_r(x)$ for all centers $x \in B_r$ and for sequence of scales $r \to 0\,.$ This can be achieved by iteratively following the particle flow generated by the mean of the velocity $u$, a construction which has been used previously for the critical equation and in conditional regularity results in \cite{CaffarelliVasseur2010,ConstantinWu2009, CH2021}. 

In contrast to the Navier-Stokes equations, we cannot rely on the classical $L^2$-based local energy inequality, but rather on a family of nonlinear $L^q$-based local energy inequalities (obtained, formally, by multiplying \eqref{eq:SQG} by $\lvert \theta \rvert^{q-2} \theta$). For $q > \frac{1}{\alpha}$, this local energy inequality has compactness  (see \cite{CH2021} for a complete discussion) which manifests itself in the possibility to carry out a a linearization argument leading to the excess decay result of  \cite{CH2021} (recalled in Proposition \ref{prop:excessdecay}  below) and to perform the aforementioned iteration argument. This restriction is responsible for not reaching the dimension bound $\frac{1}{2\alpha}(4-4\alpha)$ which is conjectured to be optimal in \cite{CH2021}; we believe that one should change the linearization in order to be able to exploit the classical $L^2$-based local energy inequality in an excess decay argument. Moreover, since the our $\varepsilon$-regularity Theorem \ref{thm:running} below requires a certain scaling-invariant smallness at all scales, the covering argument does not allow to estimate the box-counting dimension of the singular set and we therefore leave open whether the box-counting dimension estimate of \cite{CH2021} can be improved too. The same issue is observed for related results of \cite{CaffarelliKohnNirenberg1982,CDLM17} on the Navier-Stokes equations and has been the subject of many works \cite{RobinsonSadowski07,Kukavica09, KukavicaPei12,KohYang16,MR4030399, MR4030593}.

Since the overall strategy is easy to follow, but to prove the single steps several technicalities enter the game, we decided to present the flow of ideas and to introduce the relevant quantities in the next section, leaving the detailed proofs to the remaining sections.

\section{Strategy of the proof}\label{s:strategy}
\subsection{The improved $\eps$-regularity theorem} 
We introduce the scaling-invariant quantity
\begin{equation}\label{def:B}
\mathrm{B}(\theta; x, t,r ) := \frac{\norm{\theta}_{L^\infty(\R^2 \times [t-r^{2\alpha}, t+r^{2\alpha}])}^{\frac{p}{1+\alpha}-2}}{r^{\frac{p}{1+\alpha}(1-2\alpha)+2}} \int_{t-r^{2\alpha}}^{t+r^{2\alpha}} \int_0^r \int_{B_{K_q(u;t,r) r^{2\alpha - 2/q}}(x)} y^b \lvert \overline \nabla \theta^* \rvert^2 \, dz \, dy \, ds \,, 
\end{equation}
where  $\theta^*$ is the Caffarelli-Silvestre extension of $\theta$ (see Theorem \ref{thm:CF}), $q\geq 20$ should be thought to be arbitrarily large, $p> \frac{1+\alpha}{\alpha}$ is fixed and
\begin{equation}\label{def:Kq}
K_q(u;t,r):= 2 \max \{ \norm{u}_{L^\infty([t-r^{2\alpha}, t+r^{2\alpha}], L^q(\R^2))}, r^{1-2\alpha+2/q} \}\,.
\end{equation}
The quantity $\mathrm{B}$ is natural in that it involves two globally controlled quantities, the $L^\infty$-norm and the $L^2((0, \infty),W^{\alpha,2})$-norm of $\theta$. The latter is suitably localized in spacetime:  it is integrated not on a standard parabolic cylinder, but on a cylinder which, for $q= \infty$, shrinks equally fast in time and in space. This is a fundamental point related to the tilting effect of the particle flow introduced in Section \ref{s:iteration} below. 

For $\alpha \in (\frac{1}{\sqrt{6}}, \frac{1}{2})$, $p$ close to $\frac{1+\alpha}{\alpha}$ and $q$ large, the quantity $\mathrm{B}$ can be made arbitrarily small at all sufficiently small scales $r$ for every smooth solution $(\theta, u)$ of \eqref{eq:SQG}--\eqref{eq:u}\,. Conversely, we have
\begin{theorem}\label{thm:running} Let $\alpha \in (\frac{1}{\sqrt{6}}, \frac{1}{2} )$ and $q\geq 20\,.$ Set $p:= \frac{1+\alpha}{\alpha}+ \frac{1}{q}\,.$ There exists $\delta:= \delta(\alpha, q)>0$ such that if $(\theta, u)$ is a suitable Leray-Hopf solution to \eqref{eq:SQG}--\eqref{eq:u} on $\R^2 \times (0, +\infty)$ with 
\begin{equation}\label{hp-eps-reg}
\limsup_{r \to 0} \mathrm{B}(\theta; x,t,r)\leq \delta \,,
\end{equation}
then  $(x,t)$ is a regular point.
\end{theorem}
\begin{remark}\label{rmk:strangegeometry} Due to Calderon-Zygmund estimates and to the 
boundedness of suitable Leray-Hopf solutions of \eqref{eq:SQG} (see Theorem \ref{thm:LerayConstWu}), $K_q(u; t,r) \leq C_q \max \{ \norm{\theta_0}_{L^2} t^{-\frac{1}{2\alpha}(1-2/q)}, 1 \}$ for $r< \min\{ 1, t^{\frac1{2\alpha} }\}$ small enough (depending on $\norm{\theta_0}_{L^2}$ and $t$). Hence,  the smallness requirement \eqref{hp-eps-reg} is implied by 
\begin{equation}
\limsup_{r \to 0}  \frac{1}{r^{\frac{p}{1+\alpha}(1-2\alpha)+2}} \int_{t-r^{2\alpha}}^{t+r^{2\alpha}} \int_0^r \int_{B_{ r^{2\alpha - 2/q}}(x)} y^b \lvert \overline \nabla \theta^* \rvert^2 \, dz \, dy \, ds \leq \frac{\delta}{M}\,,
\end{equation}
where $M=M(\alpha, q, \norm{\theta_0}_{L^2}, t)\geq 1$ is uniform on bounded sets of initial data $\theta_{0}$ in $L^2$ and on times $t$ well-separated from $0\,.$
\end{remark}

Theorem \ref{thm:main} follows from Theorem~\ref{thm:running} via a standard covering argument and by taking the limit $q \to \infty$. 

\subsection{Fundamental quantities} The proof of Theorem~\ref{thm:running} is based on the excess decay result of \cite{CH2021} which is recalled in Proposition~\ref{prop:excessdecay} below.  
At difference from Proposition~\ref{prop:excessdecay}, we will not employ interpolated spacetime norms, but only the fundamental quantities which appear in the localized energy inequality and have a better scaling. In other words, we introduce the following scaling-invariant quantities
\begin{align}
\mathcal{A}(\theta; x,t, r)&:=\frac{1}{r^{\frac{p}{1+\alpha}(1-2\alpha)+2}} \esssup_{s \in [t-r^{2\alpha}, t]} \int_{B_r(x)} \lvert \theta(z,s) - (\theta)_{Q_r(x,t)}\rvert^\frac{p}{1+\alpha} \, dz \,, \label{def:A}\\
\mathcal{B}(\theta; x,t,r) &:= \frac{\norm{\theta}_{L^\infty(\R^2 \times [t-r^{2\alpha}, t])}^{\frac{p}{1+\alpha}-2}}{r^{\frac{p}{1+\alpha}(1-2\alpha)+2}} \int_{Q_r^*(x,t)} y^b \lvert \overline{\nabla} \theta^* \rvert^2 \, dz \, dy \, ds \,, 
\end{align}
which are localized on the standard parabolic cylinder \eqref{def:cylinder} (here $(\theta)_{Q_{r}(x,t)}$ denotes the spacetime average \eqref{def:averages} on such cylinder). The quantities $\mathcal{A}$ and $\mathcal{B}$ are related to the kinetic and dissipative part of the local ``nonlinear energy of order $\frac{p}{1+\alpha}$": indeed, the local energy inequality \eqref{e:g_energy_2} of order $q$ (when tested with a cut-off in spacetime of radius $r$ centred at $(x,t)$) controls on its left-hand side
\begin{equation}
\mathcal{E}^{tot}_q(\theta;x,t,r) :=
 \frac{1}{r^{q(1-2\alpha)+ 2 }} \bigg( \esssup_{s \in [t-r^{2\alpha}, t]} \int_{B_r(x)} \lvert \theta(z,s) - (\theta)_{Q_r(x,t)}\rvert^q \, dz + \int_{Q_r^*(x,t)} y^b \big\lvert \overline{\nabla} \big[ \lvert \theta^* \rvert^\frac{q}{2}  \big] \big\rvert^2 \, dz \, dy \, ds \bigg) \,.
\end{equation}
For $q= \frac{p}{1+\alpha}\,,$ the first term corresponds to $\mathcal{A}$ and the second is controlled by $\mathcal{B}\,.$ 

Formally,  neglecting for a moment also the additional difficulties introduced by the presence of spacetime (rather than space) averages of $\theta\,,$ $\mathcal{E}^{tot}_q$ controls $\lvert \theta- (\theta)_{Q_r(x,t)} \rvert^\frac{q}{2}$ locally in $L^\infty L^2 \cap L^2 W^{\alpha,2}$ and, by interpolation and Sobolev embedding, in $L^{2(1+\alpha)}\,.$ In this way, the smallness of the local part of the $L^p$-based excess of Proposition \ref{prop:excessdecay} can be reached by requiring the smallness of $\mathcal{A}$ and $\mathcal{B}$. This heuristics is made rigorous in Lemma \ref{lem:interpol}, where, to deal with the problem of the spacetime averages and to control also the nonlocal parts of the excess, we also need to require the smallness of the tail functional $\mathcal{T}_{p}$. Here we directly introduce the tail functional for general integrabilities $p' \geq 2$ (which may vary depending on the context):
\begin{align}
\label{def:Tp'}
\mathcal{T}_{p'}(\theta; x,t, r)&:= \frac{1}{r^{{p'}(1-2\alpha) + 2\alpha}} \int_{t-r^{2\alpha}}^t \sup_{ R \geq \frac{r}{4}} \Big( \frac{r}{R} \Big)^{\sigma p'} \bigg( \, \mean{B_R(x)} \lvert \theta(z,s) - [\theta(s)]_{B_r(x)} \rvert^\frac{3}{2}  \, dz\bigg)^\frac{2p'}{3} \, ds
\end{align}
with $\sigma \in (0, 2\alpha)$ suitably chosen. Whenever $(x,t)=(0,0)\,,$ we will omit specify the dependence with respect to the center $(x,t)$.

\subsection{Decay of the ``kinetic part" of the nonlinear local energy} 
We cannot guarantee the smallness of kinetic part $\mathcal{A}$ (because it is not of integral type) and the tails $\mathcal{T}_{p}$ (because they have a worse scaling) through the smallness assumption of our $\varepsilon$-regularity theorem. The key use of the local energy inequality in an iteration à la Caffarelli-Kohn-Nirenberg consists therefore in proving that once the dissipative part $\mathcal{B}$ is small on a certain parabolic cylinder, then the kinetic part $\mathcal{A}$ decays on a smaller scale up to reaching a certain smallness. 

\begin{prop}\label{prop:energyiter} Let $\alpha \in (\frac{1}{4}, \frac{1}{2} )$, $p> \frac{1+\alpha}{\alpha}$, $\sigma \in (1-2\alpha + \frac{2\alpha}{p}, 2\alpha)$ and $K \geq 2\,.$ Let $(\theta, u)$ be a suitable Leray-Hopf solution to \eqref{eq:SQG}--\eqref{eq:Cauchy} on $\R^2 \times [t-r^{2\alpha},t]$ such that $u= \mathcal{R}^\perp \theta + f$ with $f \in L^1_{loc}(\R)$ and 
\begin{equation}\label{eq:zeroaverage}
[u(s)]_{B_\frac{r}{2}(x)}=0 \quad \text{ for all } s \in [t-r^{2\alpha},t]\,.
\end{equation}
Then for every $\mu$ small enough (only depending on $\alpha, \sigma, p, K$) 
 and any $\varepsilon \in (0,1)$ 
  there exists $\delta= \delta(\mu,p, \eps, K) >0$ such that if
 \begin{equation}\label{hps:energyiter}
\mathcal{B}(\theta;x,t, r)< \delta \,,
 \end{equation}
we have
\begin{equation}\label{eq:energyiter}
\mathcal{A}(\theta; x,t,\mu r)^\alpha + \mathcal{T}_p(\theta; x,t,\mu r)^\frac{2}{p} \leq \frac{1}{K} \left( \mathcal{A}(\theta;x,t, r)^\alpha+ \mathcal{T}_p(\theta;x,t, r)^\frac{2}{p} \right) + \eps \, .
\end{equation}
\end{prop}
The reader may first look at this proof neglecting all nonlocal contributions; this would simplify the argument and for instance one could choose $\mu= \frac{1}{4} \,.$ A requirement of smallness of $\mu$ comes in the full proof from the tail estimate (which does not make use of the equation). The smallness assumption \eqref{hps:energyiter} on the dissipative part of the nonlinear local energy will be reached with the help of \eqref{hp-eps-reg} in the proof of Theorem~\ref{thm:running}. 

Let us point out that we require the velocity to have zero average: This seems unavoidable since all fundamental quantities $\mathcal A$ and $\mathcal B$ are of differential nature and hence do not see translations in $\theta$ or $u$. On the contrary, when we write the local energy inequality for $\theta$ minus its average on a fixed spacetime cylinder, the velocity field appearing in the equation is unchanged and hence its intensity is not under control with the fundamental quantities. Since the average of $u$ on $B_r$ scales in a supercritical way, our only option, in order to keep it under control as $r \to 0$, is to require zero average at every scale.

If we didn't need to assume zero average of the velocity $u$ in Proposition~\ref{prop:energyiter}, it would be quite classical, following \cite{CaffarelliKohnNirenberg1982}, to conclude the proof of Theorem~\ref{thm:running}. Indeed, we would iteratively apply Proposition \ref{prop:energyiter} $j$ times, letting thereby quantity $\mathcal A$ decay until it is smaller than $2\eps$ at scale $\mu^j r$. Since $\mathcal{A}$ and $\mathcal{B}$ control the $L^p$-excess by interpolation, we see that the smallness requirement of the excess decay of Proposition~\ref{prop:excessdecay}  are met provided that $\eps$ is chosen suitably small. Iterating the excess decay on a sequence of scales $r \to 0$, we conclude from Campanato's theorem that $\theta$ is H\"older continuous of order $\gamma >1-2\alpha$ in a neighbourhood of $(x,t)\,,$ hence smooth.

\subsection{Iteration of the nonlinear local energy inequality}\label{s:iteration}

To guarantee the zero-mean property \eqref{eq:zeroaverage} of the velocity field along a sequence of scales shrinking to $0$, we define, for a fixed center $(x,t) \in Q_1$ and an admissible rescaling parameter $\mu\in (0, \frac{1}{4}]$, the family $\big\{(\theta_j= \theta_j^{(x,t), \mu}, u_{j}=u_j^{(x,t), \mu})\big \}_{j \geq 0}$ of suitable Leray-Hopf solutions
\begin{align}\label{def:thetajcompact}
\theta_j^{(x,t),\mu}(z,s)&=\mu^{(2\alpha-1)j} \theta(\mu^j z + R_j(s) + x, \mu^{2\alpha j} s+t) \,, \\
u_j^{(x,t),\mu}(z,s)&= \mu^{(2\alpha-1)j} \big( u(\mu^j z + R_j(s) + x, \mu^{2\alpha j} s+t) - \mu^{-2\alpha j}\dot R_j(s)\big) \label{def:ujcompact}\,,
\end{align}
where $R_j=R_j^{(x,t), \mu}$ solves the ODE
\begin{equation}\label{eq:ODERj}
\begin{cases}
\dot R_j(s)= \mu^{2\alpha j} \displaystyle \mean{B_{1/4}}  u(\mu^j z + R_j(s) +x, \mu^{2\alpha j}s +t) \\
R_j(0) =0 \,.
\end{cases}
\end{equation}
The particle flow $R_j$ is well-defined thanks to the global control $u \in L^\infty BMO$ (see Theorem \ref{thm:LerayConstWu}) and we recall its iterative construction in Section \ref{s:particleflow} below. The above construction depends in a nontrivial way from the specific choice of $(x,t)$ and $\mu$. Since many statements work with these functions for fixed $(x,t)$ and $\mu$, we choose to explicit this dependence only when these parameters vary (for instance, in Proposition \ref{prop:ex-dec} below).

By scaling-invariance \eqref{eq:scaling}, it is easy to see that for every $j\geq 0$
\begin{enumerate}[label=(\roman*)]
\item \label{flow:i} $(\theta_j,  u_j )$ is a suitable Leray-Hopf solution of \eqref{eq:SQG} that ``lives" at scale $\mu^j$ around $(x,t)\,,$
\item  \label{flow:ii} the velocity field is of the form $u_j(z,s) = \mathcal{R}^\perp \theta_j (z,s) + f_j(s)$ for $f_j \in L^1_{loc} (\R)\,,$
\item  \label{flow:iii} the velocity field verifies the zero-mean assumption at unit scale $r= \frac{1}{4}$, that is $$[u_j(s)]_{B_{1/4}} =0 \qquad  \forall s\in [-1, 0]\,.$$
\end{enumerate}

The following proposition allows to iterate Proposition~\ref{prop:energyiter} at $(x,t)$ along a sequence of smaller scales $\mu^j$, at the price of working in each iteration with a different suitable Leray-Hopf solution $\theta_{j}$ (which, as we saw in its definition, does not only include the natural rescaling of $\theta$ by $\mu^j$, but also a space translation according to a certain flow).

\begin{prop}\label{prop:energyiter2} Let $p> \frac{1+\alpha}{\alpha}$ and $\sigma \in (1-2\alpha + \frac{2\alpha}{p}, 2\alpha)\,.$ Let $\varepsilon\in (0,1)$ and let $\mu$ be admissible for Proposition~\ref{prop:energyiter}. 
Let $(x,t) \in Q_1^+$ and let $(\theta, u)$ be a suitable Leray-Hopf solution on $\R^2  \times [-4^{2\alpha}, 0]\,$, let $(\theta_j, u_j)$ as in \eqref{def:thetajcompact}, \eqref{def:ujcompact}.
There exists $\delta=\delta(\varepsilon, \mu, p) \in (0,1)$ such that if
\begin{equation}\label{eq:hypiterj}
\mathcal{B}(\theta_j; 1/2) + \mathcal{T}_\frac{p}{1+\alpha}(\theta_j; 1/2) \leq \delta \, \text{ for } j=0, \dots, j_{0} \,,
\end{equation}
then
\begin{equation}\label{eq:iterj}
\mathcal{A}(\theta_{j_{0}}; \mu)^\alpha + \mathcal{T}_p (\theta_{j_{0}}; \mu)^\frac{2}{p} \leq \frac{1}{2^{j_{0}}} \left(\mathcal{A}(\theta_0; 1/2)^\alpha + \mathcal{T}_p(\theta_0; 1/2)^\frac{2}{p}\right)+ \varepsilon \, .
\end{equation}
\end{prop}

In order to control the intensity of the flow, the smallness assumption \eqref{eq:hypiterj}, at difference from \eqref{hps:energyiter}, needs to include a nonlocal control of the tails. Such a control is necessary to allow a comparison of $\mathcal{A}(\theta_j; 1/2)$ and $\mathcal{T}_p(\theta_j; 1/2)$ with $\mathcal{A}(\theta_{j-1}; \mu)$ and $\mathcal{T}_p(\theta_{j-1}; \mu)\,.$ We choose to enforce this control in the form of $\mathcal{T}_{\frac{p}{1+\alpha}}$, since this quantity, at difference from $\mathcal{T}_{p}$, has the same scaling as $\mathcal{B}$, and we will see in Lemma \ref{lem:limsup} that our assumption on the asymptotic smallness of $\mathrm{B}$ also implies the asymptotic smallness of $\mathcal{T}_{\frac{p}{1+\alpha}}$.

\subsection{Deduction of H\"older continuity of $\theta$ via excess decay} 
With the help Proposition~\ref{prop:energyiter2}, starting at any center $(x,t)$, we can make the quantity $\mathcal A(\theta_{j_{0}}; \mu)$ and $\mathcal{T}_p(\theta_{j_{0}}; \mu)$ smaller than any fixed $\eps$, provided that we perform sufficiently many iterations, or otherwise stated, provided that we allow a sufficiently large $j_{0}$ (whose size we show to be locally uniform with respect to the center). Hence,  for larger values of $j$, we can apply the excess decay of Proposition~\ref{prop:excessdecay} below and, iterating the excess decay along the sequence of solutions $\{ \theta_{j}, u_{j}\}$ as $j \to \infty$ (corresponding to a sequence of scales $\mu^j \to 0$), we deduce the H\"older continuity of $\theta$ via Campanato's theorem by employing the joint decay of the two excesses and by keeping under control the effect of the flow. This is the content of the next proposition, whose proof follows closely the one of \cite[Theorem 5.2]{CH2021}. We  include it nevertheless since some aspects need to be changed: for instance, the switch between the two different notions of excess. The notation $Q_r^+(x,t)$ denotes the centered in time parabolic cylinder introduced in \eqref{def:cylinder}.

\begin{prop}\label{prop:ex-dec} Let $\alpha \in (\frac{1}{4}, \frac{1}{2})$, $p> \frac{1+\alpha}{\alpha}$,  $\sigma\in (1-2\alpha + \frac{2\alpha}{p}, 2\alpha)$  and $\gamma \in (0, \sigma - \frac{2\alpha}{p})$. Let $(\theta, u)$ be a suitable Leray-Hopf solution to \eqref{eq:SQG}--\eqref{eq:u} on $\R^2 \times [-4^{2\alpha}, 4^{2\alpha}]$. 
Then, there exist $j_0:=j_0(\norm{\theta}_{L^\infty_t L_x^2}, \norm{\theta}_{L^\infty_{t,x}}) \in \N$,
 $\mu:= \mu(\alpha,p, \sigma) \in (0, \frac{1}{4}]$ and $\delta:= \delta(\alpha,p, \sigma) \in (0,1)$ such that if
\begin{equation}\label{hyp:ex-dec}
\sup_{(x,t) \in Q_{\mu^{j_0}}^+} \,\sup_{j\in\{ 0,..., j_0\}}  \mathcal{B}(\theta^{(x,t), \mu}_j; 1/2) + \mathcal{T}_{p-1}(\theta^{(x,t), \mu}_j; 1/2) < \delta \,,
\end{equation}
then
\begin{equation}\label{eq:holdercont}
\theta\in C^{\beta, \frac{p-1}{p} \beta} (Q_{\mu^{j_0}}^+) \quad \text{with } \beta := \gamma - \frac{1}{p-1}\Big[ 1-2\alpha + \frac{2\alpha}{p} \Big] \,.
\end{equation}
In particular, if $\beta >1-2\alpha\,,$ then $(0,0)$ is a regular point.
\end{prop}

Finally, the proof of Theorem~\ref{thm:running} follows from this proposition by showing how to meet the smallness assumption \eqref{hyp:ex-dec} taking advantage of the smallness of $\mathrm{B}$ in \eqref{hp-eps-reg}.

\section{Preliminary definitions and results}\label{s:preliminaries}

\subsection{Notation}\label{s:notation}
We introduce the spacetime cylinders (backward or centred in time) that are adapted to the parabolic scaling \eqref{eq:scaling}
\begin{equation}\label{def:cylinder}
Q_r(x, t):=B_r(x) \times (t-r^{2\alpha}, t] \qquad \text{ and } \qquad  Q_r^+(x,t):= B_r(x) \times (t-r^{2\alpha}, t+r^{2\alpha}) \,.
\end{equation}
In the upper half-space $\R^{n+1}_+\,,$ we define $B_r^*(x):=B_r(x) \times [0, r)$ and correspondingly
$Q_r^*(x, t) := B_r^*(x) \times (t-r^{2\alpha}, t]\,. $
We will omit the center of the cylinders whenever $(x, t)=(0,0) \,.$

For bounded sets $E \subseteq \R^n \times [0, \infty)$ and $F \subseteq \R^n\,,$ we denote the space(time) averages of functions or vector fields $f$ defined on $\R^n \times [0, \infty)$ by
\begin{equation}\label{def:averages}
(f)_E := \mean{E} f(x,t) \, dx \,dt \qquad \text{ and } \qquad [f(t)]_F :=\mean{F} f(x,t) \, dx\,.
\end{equation}

For $\Omega \subseteq \R^n$, $s \in (0,1)$, $1 \leq p < \infty $ and $f \in L^p(\Omega)$ we denote the Gagliardo semi-norm by
\begin{equation}
 [f]_{W^{s,p}(\Omega)}:= \left(\int_{\Omega} \int_{\Omega} \frac{\lvert f(x)-f(y) \rvert^p}{\lvert x-y \rvert^{n +sp} } \, dx \, dy \right)^\frac{1}{p}\,.
\end{equation}
We define the spatial Sobolev spaces  $W^{s,p}(\Omega)$  of order $s$ as the space of functions where such seminorm is finite.
When $p=2$, we sometimes denote $W^{\alpha,2}$ by $H^\alpha$ and  we recall that for $\Omega=\R^n$ the Gagliardo semi-norm coincides, up to a universal constant, with the semi-norms \eqref{e:en_of_ext-CS} below.

\subsection{Caffarelli-Silvestre extension}\label{sec:CS} We recall the following extension problem. We use the notation $\overline \nabla$, $\overline \Delta$ for differential operators defined on the upper half-space $\mathbb R^{n+1}_+$.
\begin{theorem}[Caffarelli--Silvestre \cite{CaffarelliSilvestre2007}]\label{thm:CF} Let $\alpha \in (0,1)\,,$ $f \in W^{\alpha,2} (\R^n)$ and set $b:=1-2\alpha$. There is a unique solution $f^* \in W^{1,2}(\R^{n+1}_+,y^b)$ (the  ``extension'' of $f$) to the boundary value problem
\begin{equation}\label{e:extension}
\begin{cases}
	\overline{\Delta}_b \, f^*(x,y):= y^{-b} \overline{{\rm div}}\, \big(y^b \overline\nabla f^*\big)=0\,, \\
	f^*(x,0)=f(x)\, .
\end{cases}
\end{equation}
Moreover, there exists a constant $c_{n, \alpha}$, depending only on $n$ and $\alpha$, with the following properties:
\begin{itemize}
\item[(a)] $(-\Delta)^{\alpha} f (x)=-c_{n,\alpha}\lim_{y\to 0}y^b\partial_y f^* (x,y)\,,$
\item[(b)] The following energy identity holds
\begin{equation}\label{e:en_of_ext-CS}
\int_{\R^n}|(-\Delta)^{\frac{\alpha}{2}} f^2\,dx=\int_{\R^n}|\xi|^{2\alpha}|\widehat{f}(\xi)|^2\,d\xi=c_{n,\alpha}\int_{\R^{n+1}_+}y^b|\overline\nabla f^*|^2\,dx\,dy\,.
\end{equation}
\item[(c)] The following inequality holds for every extension $g \in W^{1,2} (\R^{n+1}_+, y^b)$ of $f$:
\begin{equation}\label{e:minimum_ext_caff}
\int_{\R^{n+1}_+}y^b|\overline \nabla f^*|^2\,dx\,dy \leq \int_{\R^{n+1}_+}y^b|\overline \nabla g |^2\,dx\,dy\,.
\end{equation}
\item[(d)]\label{e:poissonformula} We have the formula $f^*(x,y)= c_{n, \alpha} \,(P(\cdot, y) \ast f)(x)\,,$ where  $P(x,y):=y^{2\alpha}/( y^2 + \lvert x \rvert^2)^\frac{n+2\alpha}{2}$.
\end{itemize}
\end{theorem}

\subsection{Suitable Leray-Hopf solutions} \label{s:LH} 
In analogy with the seminal work of Leray \cite{Leray34} for the Navier-Stokes system, the existence of distributional solutions to \eqref{eq:SQG}--\eqref{eq:u} obeying \emph{the global energy inequality} is well-known.
\begin{definition}\label{def:LerayHopf} Let $\theta_0\in L^2 (\R^2)$. A pair $(\theta,u) $, with $\theta\in L^\infty ((0,T), L^2 (\R^2)) \cap L^2 ((0,T), W^{\alpha,2} (\R^2))$, is a Leray-Hopf solution of \eqref{eq:SQG}--\eqref{eq:Cauchy}  on $\mathbb R^2\times (0,T)$ if
\begin{itemize}
\item[(a)] 
 $\theta$ is a distributional solution of \eqref{eq:SQG}--\eqref{eq:Cauchy}, namely $\div u= 0$ and for any $\varphi\in C^\infty_c (\mathbb R^2\times [0,T))$
\begin{equation}
\label{eqn:weak}
\int \big(\partial_t \varphi  \theta + u \theta \cdot \nabla \phi - (-\Delta)^\alpha \varphi  \theta \big)\, dx\, dt 
= - \int \theta_0 (x) \cdot \varphi (0,x)\, dx \,,
\end{equation}
\item[(b)] the following inequality holds for $s=0$, a.e. $s\in [0,T]$ and for every $t \in (s,T)$
\begin{align}
\frac{1}{2} \int |\theta|^2 (x,t)\, dx + \int_s^t \int |(-\Delta)^{\frac{\alpha}{2}} \theta|^2 (x,\tau)\, dx\, d\tau &\leq \frac{1}{2} \int |\theta|^2 (x,s)\, dx\,.
\label{e:g_energy_2} 
\end{align}
\end{itemize}
\end{definition}
As suggested by the maximum principle in the smooth case, Leray-Hopf solutions are bounded. 
\begin{theorem}[\cite{ConstantinWu2009} Theorem 2.1]\label{thm:LerayConstWu} Let $\theta_0 \in L^2(\R^2)$ and let $(\theta, u)$ be a Leray-Hopf solution of \eqref{eq:SQG}--\eqref{eq:Cauchy}. Then there exist a universal constant, independent on $u$, such that for any $t>0\,$
\begin{equation}\label{eq:thetaLinfty}
\sup_{x \in \R^2} \lvert \theta(x,t) \rvert \leq C\norm{\theta_0}_{L^2} t^{-\frac{1}{2\alpha}}\, .
\end{equation}
In the particular case \eqref{eq:u}, we deduce that for any $t>0\,,$
$\norm{u(\cdot, t)}_{{BMO}(\R^2)} \leq C \norm{\theta_0}_{L^2}t^{-\frac{1}{2\alpha}}\, .$
\end{theorem}
Notice that in \cite{ConstantinWu2009} Theorem \ref{thm:LerayConstWu} is proven for the coupled system \eqref{eq:SQG}--\eqref{eq:u}. However, the proof of \eqref{eq:thetaLinfty} uses only the energy inequality on level sets and the assumption $\div u=0$, not  \eqref{eq:u}.

\begin{remark}[Global level set inequalities]\label{rmk:levelset} The proof of Theorem \ref{thm:LerayConstWu} relies on the validity of a global energy inequality for all level sets $\theta_{\lambda}:=( \theta-\lambda)_{+}$ with $\lambda \geq 0$ 
\begin{equation}\label{eq:levelset}
\frac{1}{2}\int_{\R^2} \theta_{\lambda}^2(x,t) \, dx +  \int_{s}^t \int_{\R^2} \lvert (-\Delta)^\frac{\alpha}{2} \theta_{\lambda}\rvert^2 (x, \tau) \, dx \, d \tau \leq \frac{1}{2} \int_{\R^2} \theta_{\lambda}^2 (x,s) \, dx
\end{equation}
for $s=0$, a.e. $s\in [0,T]$ and for every $t \in (s,T)$.  As a technical side-remark, we were not able to derive \eqref{eq:levelset} for Leray-Hopf solutions because of a lack of differentiability. We believe that the definition of Leray-Hopf solutions should \emph{a priori} be slightly modified and also require the validity of the global level set inequalities \eqref{eq:levelset} for all $\lambda \geq 0\,.$ In particular, without this augmented definition, Theorem \ref{thm:LerayConstWu} might not exclude the existence of unbounded solutions. 

At the same time,  the global level set inequalities \eqref{eq:levelset} are verified by the Leray-Hopf solutions obtained via Leray's construction: for instance, by adding a vanishing viscosity term $\varepsilon \Delta \theta$ on the right-hand side of \eqref{eq:SQG}, a Leray-Hopf solution can be obtained as the limit, for $\varepsilon \to 0$, of smooth solutions $\theta_{\varepsilon}$ obeying the level set inequalities \eqref{eq:levelset} for every $\varepsilon>0$ \cite{CH2021}. Since $f(x)= (x-\lambda)_{+}$ is Lipschitz, the strong (and weak) convergence of $\theta_{\varepsilon}$ to $\theta$ (and $(-\Delta)^{\alpha/2} \theta_{\varepsilon}$ to $(-\Delta)^{\alpha/2} \theta$  respectively),  is passed on to the level sets and allows to pass to the limit the associated level set inequality (see \cite{Thesis} for a detailed account). 
\end{remark}

At difference from the Navier-Stokes equations, 
the system  \eqref{eq:SQG}--\eqref{eq:u}  allows to formally derive infinitely many local energy inequalities of nonlinear type. Indeed, it has been observed in \cite[Section 3.3]{CH2021} that any smooth solution $\eta$ of \eqref{eq:SQG}--\eqref{eq:Cauchy} obeys for a.e. $t\in (0,T)$, all nonnegative test functions\footnote{That is, the function $\varphi$ vanishes when 
$|x| + y + |t|$ is large enough and if $t$ is sufficiently close to $0$, but it can be nonzero on some regions of $\{(x,y,t): y=0\}$.} $\varphi \in C^{\infty}_c ( \R^3_+ \times (0,T))$ with $\partial_y \varphi (\cdot,0,\cdot)= 0$ in $\R^2 \times (0,T)$ and for any $q \geq 2$
\begin{align}\label{eqn:suit-weak-tested}
\int_{\R^2} \varphi (x,0,t) &\lvert \eta\rvert^q(x,t) \, dx +4 \big(1- \tfrac{1}{q}\big)  c_{\alpha}\int_0^t \int_{\R^3_+} y^b|\overline{ \nabla} \lvert \eta^* \rvert^\frac{q}{2} |^2 \varphi \, dx \, dy \, ds  \\
	\leq&\;
	\int_0^t \int_{\R^2}\big( \lvert \eta\rvert^q \partial_t \varphi|_{y=0} + u \lvert \eta \rvert^q \cdot \nabla \varphi|_{y=0}\big)\, dx \, ds \nonumber + c_{\alpha} \int_0^t\int_{\R^3_+} y^b \lvert \eta^*\rvert^q \overline\Delta_b \varphi \, dx \, dy \, ds \,,
	\end{align}
where the constant $c_{\alpha}$ comes from Theorem~\ref{thm:CF}. We will subsequently refer to \eqref{eqn:suit-weak-tested} as \textit{the local energy inequality of order q}. 

Since linear transformations of solutions of \eqref{eq:SQG} still solve the very same equation (with unchanged velocity), the previous observation motivates the definition of so-called suitable Leray-Hopf solutions (sometimes called suitable weak solutions) as introduced in \cite[Definition 3.4]{CH2021}.

\begin{definition}\label{def:LH}
A  Leray-Hopf solution $(\theta,u)$ is called a suitable Leray-Hopf solution of \eqref{eq:SQG}--\eqref{eq:Cauchy}  if \eqref{eqn:suit-weak-tested} holds for a.e. $t\in (0,T)$, all nonnegative test functions $\varphi \in C^{\infty}_c ( \R^3_+ \times (0,T))$ with $\partial_y \varphi (\cdot,0,\cdot)= 0$ in $\R^2 \times (0,T)$, for all $q \geq 2$ and every linear transformation of the form $\eta:= (\theta-M)/L$ with  $L>0$ and $M \in \R\,.$
	
Correspondingly, we say that $\theta$ is a suitable Leray-Hopf solution of \eqref{eq:SQG}--\eqref{eq:u} if additionally \eqref{eq:u} holds. 
\end{definition}
The class of suitable Leray-Hopf solutions includes classical solutions and for $\alpha \in (0, \frac{1}{2})$, they exist (globally in time) from any $\theta_0 \in L^2(\R^2)$ \cite[Theorem 3.6]{CH2021}.

\subsection{Singular set}\label{s:singularset} We call a point $(x,t) \in \R^2 \times (0, \infty)$ a regular point of a Leray-Hopf weak solution $\theta$ of \eqref{eq:SQG}--\eqref{eq:u}  if there exists a neighbourhood of $(x,t)$ where $\theta$ is smooth. Due to the regularizing effect of the heat equation, this is equivalent to finding a neighbourhood of $(x,t)$ where $\theta$ is $C^\gamma$-H\"older continuous in space, uniformly in time,  for $\gamma >1-2\alpha$ (see \cite{ConstantinWu2008} or \cite[Lemma B.1]{CH2021} for the precise local statement).

We denote by ${\rm Reg} \, \theta$ the open set of regular points in spacetime. Correspondingly, we define the spacetime singular set ${\rm Sing} \, \theta:=[ \R^2 \times (0, \infty)] \setminus {\rm Reg} \, \theta \, .$

\subsection{Excess decay}\label{s:excessdecay} 
 Let $p \in (3, \infty)$ and $\sigma \in (0, 2\alpha)$ to be chosen later. For $r>0$ and $Q_r(x,t) \subseteq \mathbb{R}^2 \times (0, T)$, we define the excess at scale $r$ as 
\begin{align}
E(\theta, u; x, t, r)&:= E^S(\theta; x,t,r) +
E^V(u; x, t,r)+ 
E^{NL} (\theta; x, t,r) \label{def:excess}\\
 &:= \bigg(\,\mean{Q_r(x,t)} \!\!\! \lvert \theta(z,s)- (\theta)_{Q_r(x,t)} \rvert^p \, dz \, ds \bigg)^\frac{1}{p} \label{def:Es} + \bigg(\, \mean{Q_r(x,t)}\!\!\! \lvert u(z,s)- \left[u(s)\right]_{B_r(x)} \rvert^p \, dz \, ds \bigg)^\frac{1}{p} \nonumber
 \\
& \qquad+ \bigg(\,\mean{t-r^{2\alpha}}^t \sup_{R \geq \frac{r}{4}} \left(\frac{r}{R} \right)^{\sigma p} \bigg(\, \mean{B_R(x)} \lvert \theta(z,s) -  [\theta(s)]_{B_r(x)} \rvert^{ \frac{3}{2}} \, dz \bigg)^\frac{2p}{3}  \, ds \bigg)^\frac{1}{p}
 \, ,
\end{align}
 We recall the following excess decay result for suitable Leray-Hopf solutions from \cite[Proposition 4.1]{CH2021}.

\begin{prop}[Excess decay]\label{prop:excessdecay} Let $\alpha \in (0, \frac{1}{2})$, $\sigma \in (0, 2\alpha)$, {$p > \max \left\{\frac{1+\alpha}{\alpha}, \frac{2\alpha}{\sigma} \right\}$}, $c>0$ and $\gamma \in (0, \sigma -\frac{2\alpha}{p})$. Let $Q_r(x,t) \subseteq \R^2 \times (0, \infty)$ and let $(\theta, u)$ be a suitable Leray-Hopf solution to \eqref{eq:SQG}--\eqref{eq:Cauchy} such that $u(y,s) = \mathcal{R}^\perp \theta (y,s) + f(s)$ for some $f \in L^1([t-r^{2\alpha}, t])$ and
$$\left[u(s)\right]_{B_r(x)} = 0 \text{ for all } s \in [t-r^{2\alpha}, t]\,.$$
Then there exist a universal $\mu_0= \mu_0(\alpha, \sigma, p, c, \gamma)\in \left(0, \frac{1}{2}\right)$ such that the following holds: for every $\mu \in (0, \mu_0]$, there exists $\varepsilon_0= \varepsilon_0(\alpha, \sigma, p, c, \gamma, \mu) \in (0, \frac{1}{2})$ such that, if $E(\theta, u; x, t,r) \leq r^{1-2\alpha} \varepsilon_0\,,$ then the excess decays at scale $\mu$, that is
\begin{equation}\label{eq:excessdecay}
E(\theta, u; x,t, \mu r) \leq c {\mu}^\gamma E(\theta, u; x,t,r) \, .
\end{equation}
\end{prop}
\begin{remark} In \cite{CH2021}, \eqref{eq:excessdecay} was stated only at scale $\mu=\mu_0$ (and not for all scales below $\mu_0$); as it is classical for excess decay results, the proof actually gives this stronger statement.
\end{remark}

\subsection{Tools to control the excess}\label{s:controlexcess}  Let $\alpha \in (0, 1)$ and $q \in [2, \frac{2}{1-\alpha}]$. There exist universal constants, depending only on $\alpha$ and $n$, such that for every $r>0$ and for every $f \in W^{\alpha, 2}(\R^n)$
\begin{equation}\label{eq:Poincare}
\bigg(\, \mean{B_{r}(x)} \lvert f - [f]_{B_{r}(x)} \rvert^q \, dz \bigg)^\frac{1}{q} \lesssim r^{\alpha - \frac{n}{2}} [f]_{W^{\alpha, p}(B_{r}(x))} \, \lesssim r^{\alpha - \frac{n}{2}} \bigg ( \int_{B_{4r/3}^*(x)} y^b \lvert \overline \nabla f^* \rvert^2  \, dz \, dy \bigg)^\frac{1}{2} \, .
\end{equation}
The first inequality is the Poincar\'e inequality, whereas the second one is due to a rewriting of the Gagliardo semi-norm (see  \cite[Lemma 2.3]{CH2021}). We also recall a nonlinear Poincar\'e inequality of parabolic type \cite[Lemma 6.2]{CH2021} and a control on $L^q$-excess of $u$ in terms of the excess and tails of $\theta$ \cite[Lemma 6.3]{CH2021}.  In both estimates, we change the enlarging factor of the radii to our convenience.

\begin{lemma} \label{lem:nonlinearPoincare} Let $\alpha \in (0,1) \,$ and let $(\theta, u)$ be a Leray-Hopf weak solution of \eqref{eq:SQG}--\eqref{eq:Cauchy}. We assume that $u(z,s)= \mathcal{R}^\perp \theta (z,s)  + f(s)$ for $f \in L^1_{loc}(\R)$ and that $\left[u(s)\right]_{B_r(x)} = 0$ for all $s \in [t-(2r)^{2\alpha}, t]$. 

There exists  $C=C(\alpha)\geq 1$ such that for any $q \in \big[2, \frac{2}{1-\alpha} \big]$ and for any $r>0$
\begin{align}
\frac{1}{r^{(1-2\alpha) + \frac{2}{q} + \alpha}} &\bigg(\int_{t- r^{2\alpha}}^{t} \bigg(\int_{B_r(x)} \lvert \theta - (\theta)_{Q_r(x,t)} \rvert^q \, dz \bigg)^\frac{2}{q} \, ds \bigg)^\frac{1}{2} \\
&\leq  C \Bigg (1+\bigg( \frac{1}{r^{2(1-2\alpha)+2+2\alpha}} \int_{Q_{4r/3}(x,t)} \lvert u - [u(s)]_{B_{4r/3}(x)} \rvert^2 \, dz \, ds \bigg)^\frac{1}{2} \Bigg) \, \mathcal{E}(\theta; x,t, 2r)^\frac{1}{2} \,,
\end{align}
where $\mathcal{E}(\theta;x,t,r)$ is a suitably localized version of the dissipative part of the total energy
\begin{equation}\label{def:curlyE}
\mathcal{E}( \theta; x,t,r) := \frac{1}{r^{2(1-2\alpha)+2}} \int_{Q_r^*(x,t)} y^b \lvert \overline \nabla \theta^* \rvert^2 \, dz \, dy \, ds \,. \\
\end{equation}
\end{lemma}

\begin{lemma}\label{lem:excessofu}  Let $\alpha \in (\frac{1}{4}, \frac{1}{2}),$ $q\geq 2 $ and  $\sigma \in (0, 1).$ Let $\theta \in L^q(\R^2 \times [t-(4r/3)^{2\alpha}, t])$ and consider a velocity field of the form $u(z,s)= \mathcal{R}^\perp \theta (z,s) + f(s)$ for $f \in L^1_{loc}(\R)\,.$  

There exists $C=C(q,\sigma )\geq 1$, uniform in $q$ and $\sigma$ away from $\sigma=1$ and $q=\infty$, such that
\begin{align}\label{eq:excessofu}
\frac{1}{r^{1-2\alpha}} &\bigg(\, \mean{Q_r(x,t)} \lvert u - [u(s)]_{B_r(x)} \rvert^q \, dz \, ds \bigg)^\frac{1}{q} \\
&\leq C \Bigg( \frac{1}{r^{1-2\alpha}} \bigg( \, \mean{Q_{{4r/3}}(x,t)} \lvert \theta- [\theta(s)]_{B_{{4r/3}}(x)} \rvert^q \, dz \,ds \bigg)^\frac{1}{q}+  \big(\mathcal{T}_q (\theta; x, t,  4r/3)\big)^\frac{1}{q} \Bigg)\,,
\end{align}
where the tail functional $\mathcal{T}_q$ is defined in \eqref{def:Tp'}.
\end{lemma}

\subsection{Particle flow}\label{s:particleflow} A peculiar feature of the SQG equation is that the velocity can be forced, up to a subtracting a function of time, to have zero average on a fixed ball, uniformly in time. This is achieved by following the particle flow generated by the mean of the velocity. Thanks to the scaling-invariance \eqref{eq:scaling}, it suffices to construct the particle flow at unit-scale.

\begin{lemma}\label{lem:changeofvar} Let $\alpha \in (\frac{1}{4}, \frac 12)\,,$ $q\geq 2$ and $\sigma \in (0, 1)$. Let $(\theta, u)$ be a suitable Leray-Hopf solution of \eqref{eq:SQG} on $\R^2 \times [-4^{2\alpha},0]$. 
For a fixed center $(x, t) \in Q_1$, the particle flow $\Phi^{(x,t)}(u; \cdot)$ generated by the spatial average of $u$ on $B_{1/4}(x)$ at time $t$, is well-defined as the unique solution of 
\begin{equation}\label{eq:ODE}
\begin{cases}
 \frac{d}{ds} \, \Phi^{(x,t)}(u;s)&=\displaystyle \mean{B_\frac{1}{4}} u(z+\Phi^{(x,t)}(u;s) +x, s+t) \,dz \\
\Phi^{(x,t)}(u;0)&=0 \,
\end{cases}
\end{equation}
since the vector field in the right-hand side of the ODE is log-Lipschitz.
Correspondingly, if we set
\begin{align}
\Theta^{(x,t)}(z,s)&:=\theta(z+\Phi^{(x,t)}(u;s) +x,s+t)\,, \\
 U^{(x,t)}(z,s)&:= u(z+\Phi^{(x,t)}(u;s)+x, s+t) - \dot \Phi^{(x,t)}(u;s)\,,
\end{align}
then  $\big(\Theta^{(x,t)}, U^{(x,t)}\big)$ is a suitable Leray-Hopf solution of \eqref{eq:SQG} on $\R^2 \times [-1, 0]$ with 
$$[U^{(x,t)}(s)]_{B_{1/4}}=0 \text{ for } s \in[-1,0]\, .$$ 
 Moreover, there exists universal $\varepsilon_1=\varepsilon_1(q, \alpha) \in (0,\frac{1}{2})$ and $C_1(\alpha, p) \geq 1$  such that if 
\begin{equation}\label{eq:changeofvarhyp}
\bigg(\, \mean{Q_1(x,t)} \lvert u \rvert^q \,dz \, ds\bigg)^\frac{1}{q} \leq \varepsilon_1\, ,
\end{equation}
then 
\begin{equation}\label{eq:controlonflow}
\lvert \Phi^{(x,t)}(u;s)\rvert \leq \frac{1}{4} \qquad \text{ for } s \in [-1,0]\,.
\end{equation}
Finally, the control on the flow \eqref{eq:controlonflow}, allows to relate the fundamental quantities of $(\Theta^{(x,t)}, U^{(x,t)})$ with the ones of $(\theta, u).$ More precisely, there exists $C_1(\alpha, p) \geq 1$ universal such that 
\begin{align} 
E(\Theta^{(x,t)}, U^{(x,t)} ; 1/4) &\leq C_1 E(\theta, u;x,t,1) \label{eq:changeofvarexcess}\,,\\
\mathcal{A}(\Theta^{(x,t)}; 1/2) \leq C_1 \mathcal{A}(\theta; x,t,1) \qquad &\text{ and } \qquad \mathcal{T}_p(\Theta^{(x,t)}; 1/2) \leq C_1\mathcal{T}_p(\theta;x,t, 1) \, . \label{eq:changeofvar}
\end{align}
\end{lemma}
%
%
%

The lemma follows from \cite[Lemma 5.1]{CH2021} and collects useful properties of the flow; several of them are simple consequences of the definitions. For instance, under the assumption \eqref{eq:changeofvarhyp} the vector field in the right hand side of \eqref{eq:ODE} is uniformly small, hence the flow does not move much in time $1$ as stated in \eqref{eq:controlonflow}. Moreover, \eqref{eq:changeofvarexcess}-- \eqref{eq:changeofvar} follows immediately from \eqref{eq:controlonflow} since the two excesses are then comparable; notice that the enlarging factor in \eqref{eq:changeofvarexcess}--\eqref{eq:changeofvar} can be chosen arbitrarily above $1$, provided that $\varepsilon_1$ is chosen sufficiently small in \eqref{eq:controlonflow}. We don't give the details. At difference from \cite{CH2021}, we will impose the quantitative control on the tilting effect of the flow \eqref{eq:changeofvarhyp} in terms of the $L^{p/(1+\alpha)}(Q_1(x,t))$-norm of the velocity (i.e. $q=p/(1+\alpha)$ in \eqref{eq:changeofvarhyp}).

The power of Lemma \ref{lem:changeofvar} lies in the fact that it allows to produce, from a given suitable Leray-Hopf solution of \eqref{eq:SQG}--\eqref{eq:u}, for every fixed center $(x,t)$ and scale $\mu$ a sequence of suitable Leray-Hopf solutions  $ \{\theta_j=\theta_j^{(x,t), \mu}, u_j= u_j^{(x,t), \mu} \}_{j \geq 0}$ of \eqref{eq:SQG} such that \ref{flow:i}--\ref{flow:iii} (see Section \ref{s:iteration}) hold. Each element of the sequence can be defined in two ways: either through \eqref{def:thetajcompact}, \eqref{def:ujcompact}, \eqref{eq:ODERj}, or inductively. We proceed in this second way in the following, since in iteration arguments it is more natural to drop one scale at a time. Indeed, the iterative construction is as follows. Fix a center $(x,t) \in Q_1$ and a scale $\mu \in (0, \frac{1}{4})$. 

\textit{The step $j=0$:} we set, with the notation of Lemma \ref{lem:changeofvar}, $x_0(s):=\Phi^{(x,t)}(u; s)$ and
\begin{align}
\theta_0^{(x,t), \mu} (z,s)&:= \theta(z+ x_0(s)+ x, s+t) \\
u_0^{(x,t), \mu} (z,s)&= u(z+x_0(s)+ x, s+t) - \dot x_0(s)\,.
\end{align}
Observe that this first step is independent of the choice of $\mu$; it just translates $(x,t)$ to the origin $(0,0)$ and it produces a new suitable Leray-Hopf solution whose velocity $u_0$ satisfies \ref{flow:i}--\ref{flow:ii}. Hereafter, the process will always be centred in $(0, 0)\,.$

\textit{The iterative step:} Assume that we have constructed $\theta_{j-1}^{(x,t), \mu}$ obeying \ref{flow:i}--\ref{flow:iii} , let us build $\theta_j^{(x,t), \mu}$  for some $j \geq 1$. We rescale $\theta_{j-1}$ and $u_{j-1}$ by $\mu$ according to \eqref{eq:scaling}, that is 
\begin{equation}\label{eq:theta_jrescaled}
\theta_{j-1, \mu}(z,s) := \mu^{2\alpha-1} \theta( \mu z, \mu^{2\alpha} s ) \qquad \text{ and } \qquad u_{j-1, \mu}(z,s):= \mu^{2\alpha-1} u_{j-1}(\mu z , \mu^{2\alpha} s)\,.
\end{equation}
Next, we apply Lemma \ref{lem:changeofvar} to $(\theta_{j-1, \mu}, u_{j-1, \mu})$ at the point $(0,0)$: in other words, defining $x_j(s):= \Phi^{(0,0)}(u_{j-1, \mu}; s)$, we set
\begin{align}\label{eq:thetaj-1mu}
\theta_j^{(x,t), \mu} (z,s)&:= \theta_{j-1, \mu}(z+ x_j(s), s) = \mu^{2\alpha-1} \theta_{j-1}(\mu z + \mu x_j(s), \mu^{2\alpha } s)\,, \\ \label{eq:uj-1mu}
u_j^{(x,t), \mu} (z,s)&= u_{j-1, \mu} (z+x_j(s),s) - \dot x_j(s)= \mu^{2\alpha-1} u_{j-1}(\mu z + \mu x_j(s) , \mu^{2\alpha} s) - \dot x_j(s) \,.
\end{align}
The validity of \ref{flow:i}--\ref{flow:iii} is an immediate consequence of Lemma \ref{lem:changeofvar} and the inductive hypothesis 
Finally, it is easy to verify inductively that $\theta_j$ and $u_j$ are of the form \eqref{def:thetajcompact} and \eqref{def:ujcompact}, if we set 
\begin{equation}\label{def:Rj}
R_{j}^{(x,t), \mu}(s)=R_j(s):= \sum_{k=0}^j \mu^k x_k( \mu^{2\alpha(k-j)} s)\,,
\end{equation}
and that $R_j$ solves \eqref{eq:ODERj}.

\section{Decay of the kinetic part of the nonlinear local energy: proof of Proposition~\ref{prop:energyiter}}
We introduce the following scaling-invariant analogues of the excess quantities $E^S$ and $E^V$ (defined in \eqref{def:Es} and \eqref{def:excess}): 
\begin{align}
\mathcal{C}(\theta;x,t,r)&:= \frac{1}{r^{p(1-2\alpha) + 2 +2\alpha}} \int_{Q_r(x,t)} \lvert \theta - (\theta)_{Q_r(x,t)} \rvert^p \, dz \, ds\,, \label{def:C}\\
\mathcal{D}(u; x,t,r)&:= \frac{1}{r^{p(1-2\alpha)+2+2\alpha}} \int_{Q_r(x,t)} \lvert u -[u]_{B_r(x)} \rvert^p \, dz \, ds \,.\label{def:D}
\end{align}

We prove a preliminary ``interpolation" result which controls $\mathcal{C}$ and $\mathcal{D}$ in terms of $\mathcal{A}$, $\mathcal{B}$ and $\mathcal{T}_p$. It is not a mere interpolation result, but uses already the equation through Lemma \ref{lem:nonlinearPoincare}.

\begin{lemma}\label{lem:interpol} Let $\alpha \in \big(\frac{1}{4}, \frac{1}{2}\big)$, $p> \frac{1+\alpha}{\alpha}$ and  $\sigma \in (0, 2\alpha) \,.$ There exists a universal $C_2=C_{2}(p) \geq 1$, uniform in $p$ away from $p= \infty$,  such that for every  Leray-Hopf solution $(\theta, u)$  of \eqref{eq:SQG}--\eqref{eq:Cauchy} with velocity field of the form $u= \mathcal{R}^\perp \theta +f$ for some $f \in L^1_{loc}(\R)$, it holds
\begin{equation}\label{e:interpol1}
\mathcal{D}(u;r) \leq C_2 \big( \mathcal{A}(\theta; 3r/2)^\alpha \mathcal{B}(\theta;3r/2) + \mathcal{T}_p(\theta; 3r/2) \big) \,.
\end{equation}
If additionally $
\left[u(s)\right]_{B_r(x)} = 0$ for all $s \in [t-(2r)^{2\alpha}, t]$, we also have
\begin{equation}\label{e:interpol2}
\mathcal{C}(\theta;r) \leq C_2 \, \mathcal{A}(\theta;2r)^\alpha \mathcal{B}(\theta;2r)\Big(1+ \mathcal{A}(\theta;2r)^{2\alpha/p} \mathcal{B}(\theta;2r)^{2/p}+ \mathcal{T}_p(\theta;2r)^{2/p} \Big) \,.
\end{equation}
\end{lemma}

\begin{proof} By translation and scaling invariance, we may assume that $(x,t)=(0,0)$ and that $r=1\,.$ In the following, the constant $C$ may change line by line. By Lemma \ref{lem:excessofu}, by H\"older and the Poincar\'e inequality \eqref{eq:Poincare} ( changing the enlarging factor of the radius to $9/8$ to our convenience)
\begin{align}
&\int_{Q_1} \lvert u - [u(t)]_{B_1} \rvert^p \, dx \, dt  \leq C \int_{Q_{4/3}} \lvert \theta - [\theta(t)]_{B_{4/3}} \rvert^p \, dx \, dt + C \mathcal{T}_p(\theta; 4/3)  \\
&\leq C \int_{-\left( \frac{4}{3}\right)^{2\alpha}}^0 \bigg( \int_{B_{4/3}}\lvert \theta - [\theta(t)]_{B_{4/3}} \rvert^\frac{p}{1+\alpha}  \bigg)^\alpha \bigg( \int_{B_{4/3}} \lvert  \theta - [\theta]_{B_{4/3}} \rvert^\frac{p}{(1+\alpha)(1-\alpha)} \, dx \bigg)^{1-\alpha} \, dt + C \mathcal{T}_p(\theta; 4/3) \\
&\leq C \Big( \mathcal{A}(\theta; 4/3)^\alpha (2 \norm{\theta}_{L^\infty(Q_{3/2})})^{\frac{p}{1+\alpha}-2} \int_{Q_{3/2}^*} y^b \lvert \overline{\nabla} \theta^* \rvert^2 \, dx \, dy \, dt + \mathcal{T}_p(\theta; 4/3)\Big) \\
&\leq C \Big( \mathcal{A}(\theta; 3/2)^\alpha \mathcal{B}(\theta; 3/2) + \mathcal{T}_p(\theta; 3/2)\Big) \, .
\end{align}
As for second inequality, we estimate by H\"older, Lemma \ref{lem:nonlinearPoincare} and by \eqref{e:interpol1}
\begin{align}
\int_{Q_1} \lvert \theta &- (\theta)_{Q_1} \rvert^p \, dx \, dt \leq \int_{-1}^0 \bigg(\int_{B_1} \lvert \theta - (\theta)_{Q_1} \rvert^\frac{p}{1+\alpha} \, dx \bigg)^\alpha \left( \int_{B_1} \lvert \theta - (\theta)_{Q_1} \rvert^\frac{p}{(1+\alpha)(1-\alpha)}  \, dx \right)^{1-\alpha} \, dt \\
&\leq C \mathcal{A}(\theta; 1)^\alpha (2\norm{\theta}_{L^\infty(Q_{1})})^{\frac{p}{1+\alpha}-2} \bigg( 1 + \int_{Q_{4/3}} \lvert u - [u(t)]_{B_{4/3}} \rvert^2 \, dx \, dt  \bigg) \mathcal{E}(\theta; 2)  \\
&\leq C\mathcal{A}(\theta; 2)^\alpha \mathcal{B}(\theta; 2) \Big(1+ \mathcal{A}(\theta; 2)^{2\alpha/p} \mathcal{B}(\theta; 2)^{2/p}+  \mathcal{T}_p(\theta; 2)^{2/p}\Big)\,.
\end{align}
\end{proof}

\begin{proof}[Proof of Proposition~\ref{prop:energyiter}]
By translation and scaling invariance, we may assume $(x,t)=(0,0)$ and $r=1$. Let $K\geq 2$ be given. Let $\mu_0\in (0, \frac{1}{4}]$ yet to be chosen. We consider a scale $\mu \in (0, \mu_0]$ and $\varepsilon \in (0,1)\,.$ We assume that $[u(t)]_{B_{1/2}}=0$ for $t \in [-1, 0]$ and $\mathcal{B}(\theta; 1) \leq \delta$
for a $\delta \in (0,1)$ yet to be chosen in function of $\mu, p, \eps, K\,.$ 

To lighten the notation in this proof, we do not indicate the dependence of all fundamental quantities of the solution $\theta$ and write, for instance, $\mathcal{A}(1)$ instead of $\mathcal{A}(\theta;1)\,.$

\textit{Step 1: We chose $\mu_0$ small enough to reduce the size of the tails $\mathcal{T}_p$ at scale $\mu\,.$}\\
By adding and subtracting $[\theta(t)]_{B_1}$ at fixed time, we have by the triangular inequality
\begin{align}
\mathcal{T}_p( \mu) 
&\leq 4^{p}  \mu^{-(p(1-2\alpha)+2\alpha)} \int_{-\mu^{2\alpha}}^0  \sup_{R \geq \frac{\mu}{4}} \left( \frac{\mu}{R} \right)^{\sigma p} \left(\mean{B_R} \lvert \theta - [\theta(t)]_{B_{1/4}} \rvert^\frac{3}{2} \, dx \right)^\frac{2p}{3} \, dt\,.
\end{align}
We split the supremum over radii $ \{ R \geq \frac{1}{4}\}$ and $\{ \frac{1}{4} \geq R \geq \frac \mu 4 \}\,.$ As for the first one, we have 
\begin{align}
\int_{-\mu^{2\alpha}}^0  \sup_{R \geq \frac{1}{4}} \left( \frac{\mu}{R} \right)^{\sigma p} \left(\mean{B_R} \lvert \theta - [\theta(t)]_{B_{1/4}} \rvert^\frac{3}{2} \, dx \right)^\frac{2p}{3} \, dt &\leq  2^p \mu^{\sigma p } \, \mathcal{T}_p(1) \,.
\end{align}
As for the second supremum, we have by H\"older and the Poincar\'e inequality \eqref{eq:Poincare}, with a constant $C=C(\alpha, p)$ changing line by line, that
\begin{align}
\int_{-\mu^{2\alpha}}^0  &\sup_{\frac{1}{4} \geq R \geq \frac{\mu}{4}} \bigg( \frac{\mu}{R} \bigg)^{\sigma p} \bigg(\mean{B_R} \lvert \theta - [\theta(t)]_{B_{1/4}} \rvert^\frac{3}{2}  \, dx\bigg)^\frac{2p}{3}  \, dt
\leq 4^{\sigma p +2  }\mu^{-2} \int_{Q_{1/4}}\lvert  \theta- [\theta(t)]_{B_{1/4}} \rvert^p \, dx \, dt \\
&\leq C \mu^{-2} \int_{-\left( \frac{1}{4} \right)^{2\alpha}}^0 \bigg( \int_{B_{1/4}} \lvert \theta - [\theta(t)]_{B_{1/4}} \rvert^\frac{p}{1+\alpha} \, dx \bigg)^\alpha \bigg( \int_{B_{1/4}} \lvert \theta - [\theta(t)]_{B_{1/4}} \rvert^\frac{p}{(1+\alpha)(1-\alpha)} \, dx \bigg)^{1-\alpha} \, dt \\
&\leq C \mu^{-2} \mathcal{A}(1)^\alpha \norm{\theta}_{L^\infty(Q_{1/4})}^{\frac{p}{1+\alpha}-2} \int_{-\left( \frac{1}{4} \right)^{2\alpha}}^0 \bigg( \int_{B_{1/4}} \lvert \theta - [\theta(t)]_{B_{1/4}} \rvert^\frac{2}{1-\alpha} \, dx \bigg)^{1-\alpha} \, dt\\
&\leq C \mu^{-2} \mathcal{A}(1)^\alpha \mathcal{B}(1) \, .
\end{align}
Using also Young (with exponents $(p/2, p/(p-2) )$), we conclude that 
\begin{align}
\mathcal{T}_p( \mu)^\frac{2}{p} &\leq 8 \mu^{2(\sigma-(1-2\alpha))-\frac{4\alpha}{p}} \mathcal{T}_p( 1)^\frac{2}{p} + C \mu^{-(2(1-2\alpha)+\frac{4(1+\alpha)}{p})} \mathcal{A}(1)^{\frac{2\alpha}{p}} B(1)^\frac{2}{p} \nonumber
\\ 
&\leq \frac{1}{4K} \mathcal{T}_p( 1)^\frac{2}{p} + C \mu^{-(2(1-2\alpha)+\frac{4(1+\alpha)}{p})} \mathcal{A}(1)^{\frac{2\alpha}{p}} \mathcal{B}(1)^\frac{2}{p}  \nonumber\\
&\leq  \frac{1}{4K} \mathcal{T}_p( 1)^\frac{2}{p} +\frac{1}{4K} \mathcal{A}(1)^\alpha + C K^{-1} \Big( \mu^{-(2(1-2\alpha)+\frac{4(1+\alpha)}{p})} K \mathcal{B}(1)^\frac{2}{p}\Big)^\frac{p}{p-2}\nonumber \\
&\leq  \frac{1}{4K} \Big(\mathcal{T}_p( 1)^\frac{2}{p} + \mathcal{A}(1)^\alpha \Big) + C \mu^{-8} K^2 \delta^\frac{2}{p-2} \,,\label{eq:iter1}
\end{align}
where in the second line we have chosen $\mu_0$ small enough, only depending on $K, \alpha, \sigma$ and $p$.

\textit{Step 2: We use the local energy inequality of order $\frac{p}{1+\alpha}$ to reduce the size of $\mathcal{A}(\theta; \mu)$.} \\
We use the local energy inequality with $\eta:= \theta- (\theta)_{Q_{1/4}}\,,$ $q=p/(1+\alpha) \,$ and a test function $\varphi \in C^\infty_c(\R^3_+ \times (0, \infty))$ such that $\varphi \equiv 1$ on $Q_{3/8}^*$, $\supp  \varphi\subset Q_{1/2}^*$, $\varphi$ constant in $y$ for small $y\,$ and $\norm{\varphi}_{C^2}$ smaller than a universal constant. Using H\"older to bound the nonlinear term and that $\eta^* = \theta^*-(\theta)_{Q_{1/4}}$,  we obtain, by varying $t \in [-(1/4)^{2\alpha}, 0]$, the following bound with $C=C(\alpha, p) \geq 1\,:$
\begin{align}
\mathcal{A}( 1/4) \leq C \Bigg( &\int_{Q_{1/2}} \lvert \theta - (\theta)_{Q_{1/4}} \rvert^\frac{p}{1+\alpha} + \bigg(\int_{Q_{1/2}} \lvert \theta - (\theta)_{Q_{1/4}} \rvert^p \bigg)^\frac{1}{1+\alpha} \bigg( \int_{Q_{1/2}} \lvert u - [u]_{B_{1/2}} \rvert^\frac{1+\alpha}{\alpha} \bigg)^\frac{\alpha}{1+\alpha} \nonumber \\
&+ \int_{Q_{1/2}^*} y^b \lvert \theta^* - (\theta)_{Q_{1/4}} \rvert^\frac{p}{1+\alpha} \Bigg) \nonumber \,.
\end{align}
Adding and subtracting $(\theta)_{Q_{1/4}}$ and using H\"older once more (and the assumption $p> \frac{1+\alpha}{\alpha}$), we deduce that 
\begin{align}
\mathcal{A}( \mu)^\alpha
&\leq C \mu^{-[\frac{p}{1+\alpha}(1-2\alpha)+2]\alpha} \Big(\mathcal{A}( 1/4)  + \lvert (\theta)_{Q_\mu} - (\theta)_{Q_{1/4}} \rvert^\frac{p}{1+\alpha} \Big)^\alpha \\
&\leq C \mu^{-[\frac{p}{1+\alpha}(1-2\alpha)+4]\alpha} \mathcal{A}( 1/4) ^\alpha\\
&\leq C \mu^{-[\frac{2p}{5}+4]\alpha} \mathcal{A}( 1/4)^\alpha
\\
& \leq C  \mu^{-[\frac{2p}{5}+4]\alpha}
 \Bigg( \bigg(\int_{Q_{1/2}} \lvert \theta - (\theta)_{Q_{1/2}} \rvert^\frac{p}{1+\alpha} \bigg)^\alpha +\mathcal{C}(1/2)^\frac{\alpha}{1+\alpha}\mathcal{D}( 1/2)^\frac{\alpha}{p} \\
 &\qquad \qquad \qquad+ \bigg(\int_{Q_{1/2}^*} y^b \lvert \theta^* - (\theta)_{Q_{1/2}} \rvert^\frac{p}{1+\alpha} \bigg)^\alpha \Bigg) =: \text{I}+ \text{II}+ \text{III} \, .
\end{align}
By Young's convolution inequality and Theorem \ref{thm:CF}, we have, using that the $L^1$-norm of the Poisson kernel $P$ is normalised to $1$, that 
\begin{equation}\label{eq:linfty}
\norm{\theta^*}_{L^\infty(Q_{1/2}^*)}\leq \norm{\theta^*}_{L^\infty(\R^3_+ \times [-(1/2)^{2\alpha}, 0])} \leq \norm{\theta}_{L^\infty(\R^2 \times [-1, 0])} 
\end{equation}
{and hence, after adding and subtracting $\theta$ in $\text{III}$ and using that the weight $y^b$ is locally integrable, we can estimate with $C=C(\alpha, p)\geq 1$
\begin{equation}
\text{I} + \text{III} 
\leq C \mu^{-[\frac{2p}{5}+4]\alpha} 
\Bigg(  \bigg(\int_{Q_{1/2}} \lvert \theta - (\theta)_{Q_{1/2}} \rvert^\frac{p}{1+\alpha} \bigg)^\alpha + \norm{\theta}_{L^\infty(\R^2 \times [-1, 0])}^{[\frac{p}{1+\alpha}-2]\alpha} \bigg(\int_{Q_{1/2}^*} y^b \lvert \theta^* - \theta \rvert^2 \bigg)^\alpha  \Bigg)\,.
\end{equation}
By H\"older, the first term on the right-hand side is estimated by $C(\theta; 1/2)^\frac{\alpha}{1+\alpha}$ which in turn is controlled by Lemma \ref{lem:interpol}.
As for the second term,
we write $\theta^*(x,y,t)= \theta(x,t) + \int_0^y \partial_z \theta^*(x,z,t) \, dz$ and hence we estimate by H\"older
\begin{align}
 \int_{Q_{1/2}^*} y^b \lvert \theta^*- \theta \rvert^2 \, dx \, dy \, dt  &= 
 \int_{Q_{1/2}^*} y^b \left\lvert \int_0^y \partial_z \theta^*(x,z,t) \, dz \right\rvert^2  \, dx \, dy \, dt \\ 
 &\leq \int_{Q_{1/2}^*} y^b \left( \int_0^y z^b \left(\partial_z \theta^*(x,z,t)\right)^2 \, dz \right) \left( \int_{0}^y z^{-b} \, dz \right)  \, dx \, dy \, dt \leq C \mathcal{E}( 1)
\end{align}
for some $C=C(\alpha) \geq 1 \,.$ 
We collect terms and use Young, 
 using that $p/(p-2\alpha) \in [1, 2]$ and $1/(1-\alpha) \in [4/3, 2]$, to obtain
\begin{align}
\text{I} + \text{III} &\leq C  \mu^{-[\frac{2p}{5}+4] \alpha } \left(  \left(\mathcal{A}(1)^\alpha \mathcal{B}(1)\right)^\frac{\alpha}{1+\alpha} +   \left(\mathcal{A}(1)^{\alpha} \mathcal{B}(1)\right)^\frac{\alpha(1+ 2/p)}{1+\alpha}+ \left(\mathcal{A}(1)^\alpha \mathcal{B}(1)\mathcal{T}_p(1)^\frac{2}{p} \right)^\frac{\alpha}{1+\alpha} +\mathcal{B}(1)^\alpha \right) \\
&\leq \frac{1}{4K} \left( \mathcal{A}(1)^\alpha + \mathcal{T}_{p}(1)^\frac{2}{p} \right) + CK^{-1} \bigg(  \left(K\mu^{-[\frac{2p}{5}+4] \alpha }\mathcal{B}(1)^\frac{\alpha}{1+\alpha}\right)^{1+\alpha}  \\
&\quad+ \left(K\mu^{-[\frac{2p}{5}+4]\alpha }\mathcal{B}(1)^\frac{\alpha(1+ 2/p)}{1+\alpha}\right)^\frac{1+\alpha}{1-2\alpha/p}  +  \left(K\mu^{-[\frac{2p}{5}+4]\alpha}\mathcal{B}(1)^\frac{\alpha}{1+\alpha} \right)^\frac{1+\alpha}{1-\alpha}  \bigg)+ C \mu^{-[\frac{2p}{5}+4]\alpha}  \mathcal{B}(1)^\alpha  \\
&\leq  \frac{1}{4K} \left( \mathcal{A}(1)^\alpha + \mathcal{T}_{p}(1)^\frac{2}{p} \right) + C  \mu^{-[\frac{3p}{5}+6]} K^2 \delta^\frac{1}{4} \,.
\end{align}
}
We now estimate the nonlinear term II using Lemma \ref{lem:interpol}. 
\begin{align}
\text{II} &\leq C \mu^{-[\frac{2p}{5}+4] \alpha} \Big(\mathcal{A}(1)^\alpha \mathcal{B}(1) \big( 1+ \mathcal{A}(1)^\frac{2\alpha}{p} \mathcal{B}(1)^\frac{2}{p} + \mathcal{T}_p(1)^\frac{2}{p} \big)\Big)^\frac{\alpha}{1+\alpha} \Big(\mathcal{A}(3/4)^\alpha \mathcal{B}(3/4)+ \mathcal{T}_p(3/4) \Big)^\frac{\alpha}{p} \\
 &\leq C \mu^{-[\frac{2p}{5}+4]\alpha} \Big((\mathcal{A}(1)^\alpha \mathcal{B}(1))^{\frac{\alpha(p+1+\alpha)}{p(1+\alpha)}}+(\mathcal{A}(1)^\alpha \mathcal{B}(1))^\frac{\alpha}{1+\alpha}\mathcal{T}_p(1)^\frac{\alpha}{p}+ (\mathcal{A}(1)^\alpha \mathcal{B}(1))^\frac{\alpha(p+3+\alpha)}{p(1+\alpha)}\\
 &+(\mathcal{A}(1)^\alpha \mathcal{B}(1))^{\frac{\alpha(p+2)}{p(1+\alpha)}}\mathcal{T}_p(1)^\frac{\alpha}{p}
 + (\mathcal{A}(1)^\alpha \mathcal{B}(1))^\frac{\alpha(p+1+\alpha)}{p(1+\alpha)} \mathcal{T}_p(1)^\frac{2\alpha}{p(1+\alpha)}+ (\mathcal{A}(1)^\alpha \mathcal{B}(1))^\frac{\alpha}{1+\alpha}\mathcal{T}_p(1)^{\frac{\alpha(3+\alpha)}{p(1+\alpha)}} \Big) \,.
\end{align}
All terms are of the form $C \mu^{-[\frac{2p}{5}+4] \alpha} (\mathcal{A}(1)^\alpha \mathcal{B})^{\frac{1}{r_{1}}}( \mathcal{T}_{p}(1)^\frac{2}{p})^\frac{1}{ r_{2}}$ with  $\frac{1}{r_{1}} + \frac{1}{r_{2}}<1$ (allowing  formally $r_{2}=\infty$) and hence can be estimated by Young with exponents $(r_{1}, r_{2}, \frac{r_{1} r_{2}}{r_{1}r_{2}- r_{1}- r_{2}})$ by $$\frac{1}{12} (\mathcal{A}(1)^\alpha + \mathcal{T}_{p}(1)^\frac{2}{p}) + C K^{-1} \big( \mu^{-[\frac{2p}{5}+4]} K \mathcal{B}^\frac{1}{r_{1}} \big)^\frac{r_{1} r_{2}}{r_{1}r_{2}- r_{1}- r_{2}}\,. $$ 
For instance,  we have, using also 
$\frac{p(1+\alpha)}{p-\alpha p - \alpha(1+\alpha)} \leq 6$ and $ \frac{\alpha (p+1+\alpha)}{p-\alpha p - \alpha(1+\alpha)} \geq \frac{1}{3}$, that
\begin{align}
 (\mathcal{A}(1)^\alpha \mathcal{B}(1))^\frac{\alpha(p+1+\alpha)}{p(1+\alpha)} \mathcal{T}_p(1)^\frac{2\alpha}{p(1+\alpha)} &\leq \frac{\mathcal{A}(1)^\alpha + \mathcal{T}_p(1)^\frac{2}{p}}{12K}+ \frac CK \Big(\mu^{-[\frac{2p}{5}+4]\alpha} K \mathcal{B}(1)^\frac{\alpha(p+1+\alpha)}{p(1+\alpha)}\Big)^\frac{p(1+\alpha)}{p-\alpha p - \alpha(1+\alpha)} \\
&\leq \frac{1}{12K}\Big(\mathcal{A}(1)^\alpha + \mathcal{T}_p(1)^\frac{2}{p}\Big) + C \mu^{-[\frac{6p}{5}+12]} K^5 \delta^\frac{1}{4} \,.
\end{align}
Proceeding in this way for each term and carefully minimizing the final powers on $\delta$ and maximising the ones on $K$ and $\mu^{-1}$, we obtain
\begin{equation}\label{eq:iter2}
\mathcal{A}(\mu)^\alpha \leq \frac{3}{4K} \Big(\mathcal{A}(1)^\alpha+ \mathcal{T}_p(1)^\frac{2}{p}\Big) + C \mu^{-[\frac{6p}{5}+12]}K^5 \delta^\frac{1}{4} \, .
\end{equation}

\textit{Step 3: Conclusion.} Combining \eqref{eq:iter1} and \eqref{eq:iter2}, we obtain that for all $\mu \in (0, \mu_0]$
\begin{equation}
\mathcal{A}(\mu)^\alpha + \mathcal{T}_p(\mu)^\frac{2}{p} \leq \frac{1}{K}\Big(\mathcal{A}(1)^\alpha + \mathcal{T}_p(1)^\frac{2}{p}\Big) + C \mu^{-[\frac{6p}{5}+12]} K^5(\delta^\frac{1}{4} + \delta^\frac{2}{p-2}) \, .
\end{equation}
Therefore, \eqref{eq:energyiter} holds if we choose
\begin{equation}
\delta:= \min \left \{ \Big[\mu^{[\frac{6p}{5}+12]}C^{-1} K^{-5}\Big]^4,  \Big[\mu^{[\frac{6p}{5}+12]}C^{-1} K^{-5}\Big]^{\frac{p-2}{p}}, 1 \right \} \, .
\end{equation}
\end{proof}

\section{Iteration of the local energy inequality: Proof of Proposition~\ref{prop:energyiter2}} 

\begin{proof}[Proof of Proposition~\ref{prop:energyiter2}] Let $\varepsilon \in (0,1)$ be given. We let $K\geq2$ be a parameter yet to be chosen big enough. We will use Proposition \ref{prop:energyiter} iteratively with parameters $\frac \varepsilon 2$ and $K$ and we fix an admissible rescaling parameter $\mu \in (0, \frac{1}{4})\,.$
We denote by $\delta'=\delta'(K, \mu, p,\varepsilon)\in (0,1)$ the smallness parameter given by Proposition \ref{prop:energyiter}. 
 
We fix a suitable Leray-Hopf solution $(\theta, u)$ and we consider for a fixed center $(x,t) \in Q_{1}^+$, the sequence of solutions $\{\theta_{j}=\theta_{j}^{(x,t), \mu}, u_{j}=u_{j}^{(x,t), \mu}\}_{j=0}^{j_{0}}\,.$ We recall from the construction (see Section \ref{s:particleflow}) that $u_{j} = \mathcal{R}^\perp \theta_{j} + f_{j}$ and $[u_{j}(s)]_{B_{1/4}}=0$ uniformly in $s$. We assume furthermore that the smallness assumption \eqref{eq:hypiterj} holds for $\delta \in (0, \delta']$ (which we might decide to choose even smaller). Since the latter guarantees in particular that  $\mathcal{B}(\theta_j; 1/2) \leq \delta' $ for $j=0, \dots, j_{0}$, we deduce from Proposition \ref{prop:energyiter} (applied to $(\theta_{j}, u_{j})$ with radius $r=1/2$ and center $(0,0)$) that \begin{equation}
\mathcal{A}(\theta_j; \mu)^\alpha + \mathcal{T}_p (\theta_j; \mu)^\frac{2}{p} \leq \frac{1}{K}\left(\mathcal{A}(\theta_j; 1/2)^\alpha + \mathcal{T}_p(\theta_j; 1/2)^\frac{2}{p}\right)+ \frac{\varepsilon}{2} \qquad \text{ for all } j=0, \dots j_{0} \,.
\end{equation}
We are left to estimate the effect of particle flow in order to connect the previous inequality with the one at step $j-1\,.$  Recall from the construction in Section \ref{s:particleflow} that
$$\theta_{j}(z,s)=  \theta_{j-1, \mu}(z + x_{j}(s), s) \qquad u_{j}(z,s) = u_{j-1, \mu}(z+ x_{j}(s), s)- \dot x_{j}(s)\,,$$
where the particle flow $x_{j}(s)= \Phi^{(0,0)}(u_{j-1, \mu};s)$ is obtained from the rescaling $u_{j-1, \mu}$ (as defined in \eqref{eq:uj-1mu}) by means of Lemma \ref{lem:changeofvar}. Since $[u_{j-1}]_{B_{1/4}}=0$ uniformly in time, we estimate by Lemma \ref{lem:excessofu} and Poincar\'e
\begin{align}
\bigg( \mean{Q_1} &\lvert u_{j-1, \mu} \rvert^\frac{p}{1+\alpha} \, dz \, ds \bigg)^\frac{1+\alpha}{p} = \mu^{2\alpha-1}\bigg( \mean{Q_\mu} \lvert u_{j-1} \rvert^\frac{p}{1+\alpha} \, dz \, ds \bigg)^\frac{1+\alpha}{p} \\
&\leq \mu^{2\alpha-1 - (2+2\alpha)(1+\alpha)/p} \bigg( \mean{Q_{1/4}} \lvert u_{j-1} -[u_{j-1}(s)]_{B_{1/4}} \rvert^\frac{p}{1+\alpha} \, dz \, ds\bigg)^\frac{1+\alpha}{p} \\
&\leq C \mu^{2\alpha-1- (2+2\alpha)(1+\alpha)/p} \bigg( \bigg(\mean{Q_{1/3}} \lvert \theta_{j-1} - [\theta_{j-1}(s)]_{B_{1/3}} \rvert^\frac{p}{1+\alpha} \, dz \, ds \bigg)^\frac{1+\alpha}{p}+ \mathcal{T}_\frac{p}{1+\alpha}(\theta; 1/3)^\frac{1+\alpha}{p} \bigg) \\
&\leq C  \mu^{2\alpha-1 - (2+2\alpha)(1+\alpha)/p}\Big(\mathcal{B}(\theta_{j-1}; 1/2)^\frac{1+\alpha}{p} + \mathcal{T}_\frac{p}{1+\alpha}(\theta_{j-1}; 1/2)^\frac{1+\alpha}{p}  \Big) \\
&\leq C  \mu^{2\alpha-1 - (2+2\alpha)(1+\alpha)/p} \delta^\frac{1+\alpha}{p} \, ,
\end{align}
where we use the assumption \eqref{eq:hypiterj} in the last inequality.
Choosing
\begin{equation}
\delta := \min \left \{ (C \mu^{(2\alpha-1 - (2+2\alpha)(1+\alpha)/p)})^{-p/(1+\alpha)} \varepsilon_1^{p/(1+\alpha)} ,\delta' \right\}
\end{equation}
we can therefore guarantee that 
\begin{align}\label{eq:reuse}
\bigg( \mean{Q_1} \lvert u_{j-1, \mu} \rvert^\frac{p}{1+\alpha}  \, dz \, ds \bigg)^\frac{1+\alpha}{p}  \leq \varepsilon_1 \qquad \text{ for all } j=1, \dots, j_{0}
\, .
\end{align}
Thus we have by Lemma \ref{lem:changeofvar} and by scaling invariance that
\begin{equation}
\mathcal{A}(\theta_j; 1/2)^\alpha + \mathcal{T}_p(\theta_j; 1/2)^\frac{2}{p} \leq  C_1 \Big(\mathcal{A}(\theta_{j-1, \mu}; 1)^\alpha + \mathcal{T}_p(\theta_{j-1, \mu}; 1)^\frac{2}{p}\Big) =  C_1 \Big(\mathcal{A}(\theta_{j-1}; \mu)^\alpha + \mathcal{T}_p(\theta_{j-1}; \mu)^\frac{2}{p}\Big) 
\end{equation} 
for all $j=1, \dots, j_{0}\,.$ Thus choosing $K:= 2 C_1 \geq 2\,,$ we have
\begin{align}
\mathcal{A}(\theta_{j_{0}}; \mu)^\alpha + \mathcal{T}_p(\theta_{j_{0}}; \mu)^\frac{2}{p} &\leq \frac{1}{2C_1}  \Big(\mathcal{A}(\theta_{j_{0}}; 1/2)^\alpha + \mathcal{T}_p(\theta_{j_{0}}; 1/2)^\frac{2}{p} \Big) + \frac{\varepsilon}{2}  \\
&\leq \frac{1}{2}  \Big(\mathcal{A}(\theta_{j_{0}-1}; \mu)^\alpha + \mathcal{T}_p(\theta_{j_{0}-1}; \mu)^\frac{2}{p} \Big) + \frac{\varepsilon}{2} \\
&\leq \frac{1}{2^2 C_1}  \Big(\mathcal{A}(\theta_{j_{0}-1}; 1/2)^\alpha + \mathcal{T}_p(\theta_{j_{0}-1}; 1/2)^\frac{2}{p} \Big) + \frac{ \varepsilon}{4}+ \frac{\varepsilon}{2} \\
&\leq \frac{1}{2^2}  \Big(\mathcal{A}(\theta_{j_{0}-2}; \mu)^\alpha + \mathcal{T}_p(\theta_{j_{0}-2}; \mu)^\frac{2}{p} \Big) + \frac{ \varepsilon}{4}+ \frac{\varepsilon}{2}  \\
&\leq \cdots \leq \frac{1}{2^{j_{0}} C_1}   \Big(\mathcal{A}(\theta_0; 1/2)^\alpha + \mathcal{T}_p(\theta_0; 1/2)^\frac{2}{p} \Big) + \frac{\varepsilon}{2} \sum_{j=0}^{j_0-1} 2^{-j} \\
& \leq \frac{1}{2^{j_{0}}}   \Big(\mathcal{A}(\theta_0; 1/2)^\alpha + \mathcal{T}_p(\theta_0; 1/2)^\frac{2}{p} \Big) + \varepsilon \,.
\end{align}
\end{proof}

\section{Deduction of H\"older continuity: proof of Proposition~\ref{prop:ex-dec}}

\begin{proof}[Proof of Proposition~\ref{prop:ex-dec}] Let $\varepsilon \in (0,1)$ yet to be chosen small enough and $j_0 \in \N$ yet to be chosen large enough. We will apply Proposition~\ref{prop:energyiter2} with $\frac{\varepsilon}{2}$: we denote by $\mu \in (0, \frac{1}{4}]$ an admissible rescaling parameter and by $\delta=\delta(\varepsilon, \mu, p) \in (0,1)$ the associated smallness requirement given by Proposition ~\ref{prop:energyiter2}. We now choose $\mu$ small enough such that it is also an admissible scale for the excess decay of Proposition \ref{prop:excessdecay} and we denote by $\varepsilon_0$ the universal smallness requirement of the excess decay of Proposition \ref{prop:excessdecay}, by $\varepsilon_1$ and $C_1$ the constants of Lemma \ref{lem:changeofvar} and by $C_2$ the constant from Lemma \ref{lem:interpol}. We fix a suitable Leray-Hopf solution $(\theta, u)$ and assume that \eqref{hyp:ex-dec} holds with the above choice of $\delta\,.$

\textit{Step 1:  With a suitable choice of $j_0 \in \N$, depending only on global norms of $\theta$, we reach the hypothesis of the excess decay Proposition \ref{prop:excessdecay} uniformly in $Q_{\mu^{j_0}}^+\,.$}\\
Observe that for every fixed center $(x,t) \in Q_1^+$ it holds $Q_{1/2}(x,t) \subset B_{3/2} \times [-2,2 ]$. Thus, there is a universal constant $C=C(p) \geq 1$ such that for all $(x,t) \in Q_1^+$
\begin{align}
\mathcal{A}\big(\theta_0^{(x,t), \mu }; 1/2\big) + \mathcal{T}_p\big(\theta_0^{(x,t), \mu }; 1/2\big) &\leq C \norm{\theta}_{L^\infty(\R^2 \times [-2, 2])}^{p-3} \norm{\theta}_{L^\infty([-2, 2], L^2(\R^2))}^2 \,.
\end{align}
This bound is independent of $\mu$, since $(\theta_0^{(x,t), \mu }, u_{0}^{(x,t), \mu})$ is only a translation of $(\theta, u)$ according to the particle flow $x_{0}=\Phi^{(x,t)}(u;\cdot)$ which doesn't involve any rescaling yet  (see Section \ref{s:particleflow}). This allows to choose ${j_0}={j_0}(\norm{\theta}_{L^\infty(\R^2 \times [-2, 2])}, \norm{\theta}_{L^\infty([-2, 2], L^2(\R^2))}, \varepsilon)\in \N$ large enough so that 
\begin{equation}
\frac{1}{2^{{j_0}-1}} \Big( \mathcal{A}\big(\theta_0^{(x,t), \mu}; 1/2\big)^\alpha + \mathcal{T}_{p}\big(\theta_0^{(x,t), \mu}; 1/2\big) \Big) \leq \frac{\varepsilon}{2}  \quad \text{for all } (x,t) \in Q_1^+ \,.
\end{equation}
We apply Proposition \ref{prop:energyiter2} with this choice of $j_{0}$ and we infer that if \eqref{hyp:ex-dec} holds, then in particular
\begin{equation}\label{eq:smallness}
\mathcal{A}\big(\theta_{j_0-1}^{(x,t), \mu}; \mu\big)^\alpha + \mathcal{T}_p\big (\theta_{j_0-1}^{(x,t), \mu}; \mu\big)^\frac{2}{p} \leq \varepsilon \quad  \text{for all } (x,t) \in Q_{\mu^{j_0}}^+ \, .
\end{equation}
Moreover, from the proof of Proposition \ref{prop:energyiter2} (see specifically \eqref{eq:reuse}), we also have that
\begin{equation}\label{eq:controlchangevar}
  \bigg( \mean{Q_1} \big\lvert u_{j}^{(x,t), \mu}(z, s) \big\rvert^\frac{p}{1+\alpha} \, dz \, ds \bigg)^\frac{1+\alpha}{p} \leq \varepsilon_1 \quad  \text{for all } (x,t) \in Q_{\mu^j_0}^+ \text{ and for all } j=0, \dots, j_0 -1 \,.
\end{equation}
This allows now to bound the excess of $\theta_{j_{0}}^{(x,t), \mu}$ at scale $\frac{1}{4}$ uniformly in the centers $(x,t) \in Q_{\mu^{j_0}}^+$. To lighten notation, we sometimes write $\theta_{j_{0}}$ instead of $\theta_{j_{0}}^{(x,t), \mu}$ in the following three estimates.  By means of Lemma \ref{lem:interpol}, by combining \eqref{eq:controlchangevar} with Lemma \ref{lem:changeofvar} and by scale invariance, we deduce that if \eqref{hyp:ex-dec} holds, then for every $(x,t) \in Q_{\mu^{j_0}}^+$
\begin{align}
\big(4^{1-2\alpha} &E^S(\theta_{{j_0}}^{(x,t), \mu}; 1/4)\big)^p = \mathcal{C}(\theta_{{j_0}}; 1/4) \\
&\leq C_2 \mathcal{A}(\theta_{{j_0}}; 1/2)^\alpha \mathcal{B}(\theta_{{j_0}}; 1/2)\big( 1+ \mathcal{A}(\theta_{{j_0}}; 1/2)^\frac{2\alpha}{p} \mathcal{B}(\theta_{{j_0}}; 1/2)^\frac{2}{p} + \mathcal{T}_p(\theta_{{j_0}}; 1/2)^\frac{2}{p}\big) \\
&\leq  C_2 C_1^2 \mathcal{A}(\theta_{{j_0}-1, \mu};1)^\alpha \mathcal{B}(\theta_{{j_0}}; 1/2)( 1+ \mathcal{A}(\theta_{{j_0}-1,  \mu}; 1)^\frac{2\alpha}{p} \mathcal{B}(\theta_{{j_0}}; 1/2 )^\frac{2}{p} + \mathcal{T}_p(\theta_{{j_0}-1, \mu};1)^\frac{2}{p}) \\
&=  C_2 C_1^2 \mathcal{A}(\theta_{{j_0}-1}; \mu)^\alpha \mathcal{B}(\theta_{{j_0}}; 1/2)( 1+ \mathcal{A}(\theta_{{j_0}-1}; \mu)^\frac{2\alpha}{p} \mathcal{B}(\theta_{{j_0}}; 1/2 )^\frac{2}{p} + \mathcal{T}_p(\theta_{{j_0}-1}; \mu)^\frac{2}{p}) \\
&\leq 3 C_2 C_1^2 \varepsilon \, .
\end{align}

Similarly, we have for all $(x,t) \in Q_{\mu^{j_0}}^+$ (with a universal constant $C=C(p) \geq 1$) that 
\begin{align}
\big(4^{1-2\alpha} E^V(u_{{j_0}}^{(x,t), \mu}; 1/4)\big)^p &\leq C_2 (\mathcal{A}(\theta_{{j_0}}; 3/8)^\alpha \mathcal{B}(\theta_{{j_0}}; 3/8)+ \mathcal{T}_p(\theta_{{j_0}}; 3/8)^\frac{2}{p}) \\
&\leq C  C_2 (\mathcal{A}(\theta_{{j_0}}; 1/2)^\alpha \mathcal{B}(\theta_{{j_0}}; 1/2)+ \mathcal{T}_p(\theta_{{j_0}}; 1/2)^\frac{2}{p})\\
&\leq C C_1 C_2(\mathcal{A}(\theta_{{j_0}-1}; \mu)^\alpha \mathcal{B}(\theta_{{j_0}}; 1/2)+ \mathcal{T}_p(\theta_{{j_0}-1}; \mu)^\frac{2}{p}) \\
&\leq 2 C C_1 C_2 \varepsilon \,
\end{align}
and  that
\begin{align}
(4^{1-2\alpha} E^{NL}(\theta_{{j_0}}^{(x,t), \mu}; 1/4))^p \leq C \mathcal{T}_p(\theta_{{j_0}}; 1/2)\leq C C_1  \mathcal{T}_p(\theta_{{j_0}-1}; \mu) \leq C C_1 \varepsilon \, .
\end{align}
Thus, if we choose
\begin{equation}\label{eq:choiceofepsilon}
\varepsilon:= \min \left\{ \frac{\varepsilon_0^p}{3^p 3 C C_1^2 C_2},  1\right\} ,
\end{equation}
then we deduce from \eqref{hyp:ex-dec} that 
\begin{equation}\label{eq:endstep1}
E\big(\theta_{j_0}^{(x,t), \mu}; 1/4\big) \leq \varepsilon_0  \quad \text{ for all } (x,t) \in Q_{\mu^{j_0}}^+ \, .
\end{equation}

\textit{Step  2:  In what follows, we will fix a point $(x,t) \in Q_{\mu^{j_0}}^+$ and we denote $\theta_j^{(x,t), \mu}$ simply by $\theta_j\,.$ We prove the excess decay along  the sequence of solutions $(\theta_{j}, u_{j})$, i.e. we prove that for $j \geq {j_0}$ it holds
\begin{align}
E(\theta_j, u_j;\mu) &\leq C_1^{-1} \mu^{\gamma(j+1-j_0)} \mu^{(2\alpha-1)(j-j_0)} \varepsilon_0 \label{eq:ind1} \\
E(\theta_j, u_j ; 1/4) &\leq \mu^{(\gamma-(1-2\alpha))(j-j_0)} \varepsilon_0\,. \label{eq:ind2}
\end{align}
} \\
We proceed by induction on $j \geq j_0 \,.$

\textit{The case $j=j_0\,.$} We already established \eqref{eq:ind2} for $j=j_0$ in \eqref{eq:endstep1}. Since by construction $[u_{j_0}(s)]_{B_{1/4}}=0$ for $s \in [-1, 0]$ and $u_{j_0}= \mathcal{R}^\perp \theta_{j_0} + f_{j_0}$ for some $f_{j_0} \in L^1_{loc}$ (see Section \ref{s:particleflow}),  we apply Proposition \ref{prop:excessdecay} with parameters $c=C^{-1}$, $\gamma$ and $r=1/4$ and we deduce that
\begin{equation}
E(\theta_{j_0}, u_{j_0}; \mu)\leq C_1^{-1} \mu^\gamma E(\theta_{j_0}, u_{j_0}; 1/4)  \leq C_1^{-1} \mu^\gamma \varepsilon_0 \,.
\end{equation}

\textit{The inductive step.} Assume that the claim holds for some $j \geq j_0 \,.$ We need to show that it holds at $j+1 \,.$ Recall from Section \ref{s:particleflow} that 
$$\theta_{j+1}(z,s)= \theta_{j, \mu}(z + x_{j+1}(s), s) \qquad u_{j+1}(z,s)= u_{j, \mu}(z + x_{j+1}(s), s) - \dot x_{j+1}(s)\,, $$
where $x_{j+1}(s)= \Phi^{(0,0)}(u_{j, \mu}; s)$ is obtained from the rescaling $u_{j, \mu}$ by applying Lemma \ref{lem:changeofvar} at the point $(0, 0)\,.$ Since $[u_j]_{B_{1/4}}=0$ by construction, we observe that by induction hypothesis (\eqref{eq:ind2} holds at $j$) 
\begin{align}
\bigg(\mean{Q_1} \lvert u_{j, \mu} (y,s) \rvert^{p} \, dy \, ds\bigg)^\frac{1}{p} &\leq  \mu^{2\alpha-1} (4\mu)^{- \frac{2+2\alpha}{p}} \bigg(\mean{Q_\frac{1}{4}} \lvert u_{j} - [u_j(s)]_{B_\frac{1}{4}} \rvert^{p} \, dy \, ds \bigg)^\frac{1}{p} 
\leq  \mu^{2\alpha-1} (4\mu)^{- \frac{2+2\alpha}{p}}\varepsilon_0\,.
\end{align}
Up to choosing $\varepsilon_0$ even smaller such that the above doesn't exceed $\varepsilon_1$, we deduce from Lemma \ref{lem:changeofvar} (with $q=p$ this time) and the induction hypothesis (\eqref{eq:ind1} holds at $j$) that 
\begin{equation}
E(\theta_{j+1}, u_{j+1}; 1/4) \leq  C_1 E(\theta_{j, \mu}, u_{j, \mu}; 1) =  C_1 \mu^{2\alpha-1} E(\theta_{j}, u_{j}; \mu) \leq \mu^{(\gamma-(1-2\alpha))(j+1-j_0)} \varepsilon_0\,.
\end{equation}
From the above, we deduce with excess decay of Proposition \ref{prop:excessdecay} (applied with the same parameters as before) that 
 \begin{equation}
 E(\theta_{j+1}, u_{j+1}; \mu) \leq  C_1^{-1} \mu^\gamma  E(\theta_{j+1}, u_{j+1}; 1/4) \leq C_1^{-1}  \mu^{\gamma (j+2-j_0)} \mu^{(2\alpha-1)(j+1-j_0)} \varepsilon_0 \, .
 \end{equation}
 
\textit{Step 3: We deduce H\"older continuity of $\theta$ via Campanato's Theorem.}\\
We first establish a control on the tilting effect of particle flow. Let $(x,t) \in Q_{\mu^{j_0}}^+$ fixed and recall from the construction that 
\begin{equation}
\theta_j(z,s)= \mu^{(2\alpha-1)j} \theta( \mu^j z + R_j(s) +x, \mu^{2\alpha j} s + t)\,,
\end{equation}
where $R_j=R_{j}^{(x,t), \mu}$ is defined in \eqref{def:Rj}. For $j \geq j_0$ we write $R_j(s)= R_{j_0}(\mu^{2\alpha(j-j_0)} s) + r_{j > j_0}(s) \,,$ where
\begin{equation}
r_{j > j_0}(s):= \sum_{k=j_0+1}^j \mu^k x_k(\mu^{2\alpha(j-k)} s) \,.
\end{equation}
It is easy to prove that, up to requiring that $\varepsilon_0$ is smaller than $\frac{3}{4} \mu^{1-2\alpha+\frac{2}{p}}$, we have 
\begin{equation}
\norm{ \dot{x_j}}_{L^p((-1, 0))}\leq  \mu^{-(1-2\alpha+2/p)} E^V(u_{j-1}; 1/4) \leq \mu^{-(1-2\alpha+2/p)} \varepsilon_0 \qquad \text{ for all } j \geq j_{0}+1 \,.
\end{equation}
Thus we can bound 
\begin{align}
\lvert r_{j >j_0}(s ) \rvert &\leq \varepsilon_0 \mu^{-(1-2\alpha+2/p)} \lvert s \rvert^{1-1/p} \sum_{k=j_0+1}^j  \mu^k \mu^{2\alpha(j-k)(1-1/p)}\nonumber \\
&\leq C \varepsilon_0 \mu^{-2/p + 2\alpha/p}  \mu^{2\alpha j(1-1/p)} \lvert s \rvert^{1-1/p} \quad \forall s \in (-1,0)\,, \label{e:estrj}
\end{align}
where $C=C(\alpha, p, \mu):= \sum_{k=0}^\infty \mu^{k(1-2\alpha(1-1/p))}\,.$

The first $j_0$ spatial translations cannot be made small by means of the excess (which, in fact, might not be small); however, we observe that if we look at small enough times, the translations don't have enough time to act and thus can be made as small as we want thanks to the global control of $u \in L^\infty_{t} BMO_x$ (see Theorem \ref{thm:LerayConstWu}). More precisely, we make the following

\textit{Claim: there exists $ j_1=j_1(\alpha, \mu, p, \norm{u}_{L^{2p}(Q_{1}^+)}, j_{0})>j_0$ large enough such that for all $j \geq j_1$
\begin{equation}\label{e:claim}
\lvert R_{j_{0}}(\mu^{2\alpha(j-j_{0})}s) \rvert \leq \frac{1}{4}  \mu^{2\alpha j (1-1/p)} \lvert s \rvert^{1-1/p} \qquad  \forall s \in (-1, 0)  \, .
\end{equation}
}
\textit{Proof of the claim:} As long as $\lvert R_{j_{0}}(s) \rvert \leq \frac{1}{4}$, we deduce from \eqref{eq:ODERj} the estimate 
$$\lvert \dot R_{j_{0}}(s) \rvert \leq 2 \mu^{(2\alpha- \frac{1}{p})j_{0}} \left( \int_{B_{1/2}} \lvert u \rvert^{2p} (z+x, \mu^{2\alpha j_{0}} s +t) \, dz \right)^\frac{1}{2p} $$ 
and hence, recalling that $(x,t) \in Q_{\mu^{j_{0}}}^+ \subset Q_{1/4}^+$, the improved estimate
\begin{equation}\label{eq:proofofclaimR}
\lvert R_{j_{0}}(s) \rvert \leq 2 \mu^{(2\alpha- \frac{1}{p})j_{0}} \lvert s \rvert^{1-1/(2p)}  \norm{u}_{L^{2p}(Q_{1}^+)}\,.
\end{equation}
We now choose $j_{1}=j_{1}(\alpha, \mu, p, \norm{u}_{L^{2p}(Q_{1}^+)}, j_{0})>j_{0}$ large enough so that 
$$\mu^{\frac{\alpha}{p} j_{1}}\left( 2 \mu^{-\frac{(1-2\alpha)}{p}j_{0}} \norm{u}_{L^{2p}(Q_{1}^+)} \right) < \frac{1}{4} \,.$$
For every $j \geq j_{1}$, this choice guarantees that $\lvert R_{j_{0}}(s) \rvert < \frac{1}{4}$ for all $s \in [-\mu^{2\alpha(j-j_{0})}, 0]$ and hence, \eqref{e:claim}
is a consequence of \eqref{eq:proofofclaimR}.
\qed 

Combining \eqref{e:estrj} with the claim \eqref{e:claim} and up to choosing $\varepsilon_0$ smaller than $\frac{1}{2}C^{-1} \mu^{2/p -2\alpha/p}$, we have that 
\begin{equation}
\lvert R_j(s)  \rvert \leq \frac{3}{4} \mu^{2\alpha j(1-1/p)} \lvert s \rvert^{1-1/p}\, \quad \forall s \in (-1,0) \text{ and }\forall j\geq  j_1 \,.\label{e:estrjfinal}
\end{equation}
It is important to observe that $\varepsilon_{0}$ does not depend on neither $j_{0}$ nor $j_{1}$, but only on the parameters $\alpha, p, \sigma, \gamma$ and $\mu\,.$ For $j \geq j_1$, we therefore set
\begin{equation}\label{e:Ij}
I_{j+1}:= \Big( \mu^{(j+1)( \frac{p}{p-1} -2\alpha) }, 0\Big]\,,
\end{equation}
so that on $I_{j+1}$ we have the following control on the tilting effect of the particle flow:
$$ B_{1/4} \subseteq  B_1 + \mu^{-(j+1)} R_j(\mu^{2\alpha} s) \qquad  \forall s \in I_{j+1}  \,.$$
This control allows to deduce that $\theta$ belongs to some suitably chosen Campanato space. Indeed, we have for a constant $A$ yet to be chosen suitably, we have
\begin{align*}
\bigg(\mean{ I_{j+1}} \mean{ B_{1/4}} &\lvert \theta(\mu^{j+1}z +x, \mu^{(j+1)2\alpha}s +t)-A \rvert \, dz \,ds \bigg)^\frac{1}{p} \\
&\leq  4 \bigg(\mean{ I_{j+1}} \mean{ B_1}  \lvert \theta(\mu^{j+1}z + R_j (\mu^{2\alpha s}) +x, \mu^{(j+1)2\alpha}s +t)-A \rvert \, dz \,ds \bigg)^\frac{1}{p} \\
&\leq 4 \, \lvert I_{j+1} \rvert^{-1/p} \bigg(\mean{ Q_1}  \lvert \theta(\mu^{j+1}z + R_j(\mu^{2\alpha s}) +x, \mu^{(j+1)2\alpha}s +t)-A \rvert \, dz \,ds \bigg)^\frac{1}{p} \\
&=4 \mu^{-(j+1)(\frac{1}{p-1}-\frac{2\alpha}{p})} \mu^{(1-2\alpha)j} \bigg( \mean{Q_\mu} \lvert \theta_j(z,s) - A \rvert^p \, dz \, ds \bigg)^\frac{1}{p}\,. \\
\end{align*}
Choosing $A= (\theta_j)_{Q_\mu}$, we deduce from the previous estimate and \eqref{eq:ind1} that for $r= \mu^{j+1}$
\begin{align}\label{eq:Campanato}
&\bigg(\mean{t- r^{ \frac{p}{p-1}}}^t \mean{ B_{r/4}(x)} \lvert \theta - (\theta)_{B_{r/4}(x) \times (t- r^{p/(p-1)},t] }\rvert^p \, dz \, ds \bigg)^\frac{1}{p} \leq C \, r^{( \gamma -\frac{1}{p-1} [1-2\alpha + \frac{2\alpha}{p}])} \varepsilon_0
\end{align}
where we set $C:=8 C_1^{-1} \mu^{-(\gamma-(1-2\alpha))j_0}\,.$ Since \eqref{eq:Campanato} holds for all radii of the form $r= \mu^{j+1}$ for $j \geq j_1$, a standard argument shows that \eqref{eq:Campanato} holds in fact, up to enlarging the constant $C$, also for all intermediate radii $\mu^{j+2} < r < \mu^{j+1}\,.$ In other words, \eqref{eq:Campanato} holds for all $r \in (0,r_0]$ where $r_0:= \mu^{j_1+1}$ does not depend on $(x,t)$. Since $(x,t)$ was an arbitrarily chosen point in $Q_{\mu^{j_0}}^+$, we deduce by a variant of Campanato's Theorem\footnote{To be precise, we apply Campanto's Theorem respect to the metric  $d((x,t), (y,s)):= \max \{ \lvert x-y \rvert, \lvert t-s \rvert^{(p-1)/p} \}$ on spacetime) where in time, as usual for parabolic equations, we only look at backward-in-time intervals. The proof of this version of Campanato's Theorem follows, for instance, line by line \cite[Theorem 1]{Schlag1996} when replacing the parabolic metric by $d\,.$} \cite[Theorem 2.9.]{Giusti2003} that \eqref{eq:holdercont} holds. Finally, smoothness in the case that the H\"older continuity exponent exceeds $1-2\alpha$ follows for instance from \cite[Lemma B.1 ]{CH2021}.
\end{proof}

\section{$\varepsilon$-regularity result: Proof of Theorem~\ref{thm:running}}\label{s:proof eps}

In order to translate the smallness requirement of Proposition \ref{prop:energyiter2} into a smallness requirement on $\theta\,,$ we need a control on the tilting effect (in space) of the particle flow  $R_j\,.$ Recalling \eqref{eq:ODERj} and \eqref{def:ujcompact}, we estimate
\begin{align}\label{eq:boundflow}
\norm{ R_j^{(x,t), \mu}}_{L^\infty([-1, 0])} \leq\norm{ \dot R_j^{(x,t), \mu}}_{L^q([-1, 0])} &\leq 2 \mu^{(2\alpha - 2/q)j} \norm{u}_{L^\infty([t-\mu^{2\alpha j}, t], L^q(\R^2))} \nonumber \\
&\leq \mu^{(2\alpha -2/q)j} K_q(u;t, \mu^j) \,, 
\end{align}
where $K_q(u; t, r)$ is defined in \eqref{def:Kq}. Due to boundedness of the Riesz-transform on $L^q$ and the boundedness of Leray-Hopf solutions (see Theorem \ref{thm:LerayConstWu}), $ K_q(u;t, r)$ is, for small $r$, uniform on bounded sets of initial data in $L^2$ and for times well-separated from $0$ (see Remark \ref{rmk:strangegeometry}).

\begin{lemma}\label{lem:boundcurlybt} Let $\alpha\in (\frac{1}{\sqrt{6}}, \frac{1}{2})$, $p>3$, $\sigma \in (0, 2\alpha)$ and $q\geq 10$. There exists $C=C(\alpha, p, \sigma) \geq 1$ such that
\begin{align}
\sup_{(x,t) \in Q_{\mu^j}^+} \mathcal{B}(\theta_j^{(x,t), \mu}; 1/2) &\leq C \, \mathrm{B}(\theta; 4\mu^j)\,, \label{eq:boundcurlyb} \\
\sup_{(x,t) \in Q_{\mu^j}^+} \mathcal{T}_\frac{p}{1+\alpha}(\theta_j^{(x,t), \mu}; 1/2) &\leq C\, \mathrm{T}(\theta; 4\mu^j) \label{eq:boundcurlyt}
\end{align}
for $j \geq 0\,,$ where $\mathrm{B}$ is defined in \eqref{def:B} and 
\begin{align}\label{def:T} 
&\mathrm{T}(\theta;x,t, r):= \\
&\frac{1}{r^{\frac{p}{1+\alpha}(1-2\alpha)+2\alpha}} \int_{t-r^{2\alpha}}^{t+r^{2\alpha}} \,  \sup_{\lvert z-x \rvert\leq \frac{3}{4} K_q(u;t,r) r^{2\alpha-2/q}} \, \sup_{R \geq \frac{r}{4}} \left( \frac{r}{R}\right)^{ \frac{\sigma p}{1+\alpha}} \left( \mean{B_R(z)} \lvert \theta - [\theta(s)]_{B_\frac{r}{4}(z)} \rvert^\frac{3}{2} \, dz' \right)^\frac{2p}{3(1+\alpha)} \, ds \,.
\end{align}
\end{lemma}
\begin{proof} The choice of the parameters guarantees in particular that 
\begin{equation}\label{eq:stupidparameters}
2\alpha - 2/q > \frac{1}{2} \,.
\end{equation}
Fix $(x,t) \in Q_{\mu^j}^+\,.$ Using the rewriting \eqref{def:thetajcompact}, Theorem \ref{thm:CF}, the fact that the Poisson kernel $P$ is 2-homogeneous and that convolution commutes with translations, we have
\begin{equation}
\theta_j^*(z,y, s)= (P(\cdot, y) \ast \theta_j(\cdot, s))(z)  =  \mu^{(2\alpha-1)j} \theta^*(\mu^jz+ R_j(s)+x, \mu^j y, \mu^{2\alpha j} s+t)\,.
\end{equation}
Thus, for fixed times
\begin{align}
\int_{B_{\frac{1}{2}}^*} y^b \lvert \overline \nabla \theta_j^*(z,y,s) \rvert^2 \, dz \, dy  = \mu^{j(2\alpha-2(1-2\alpha)-2)} \int_0^{\frac{\mu^j}{2}} \int_{B_\frac{\mu^j}{2}(R_j(s) + x)} y^b \lvert \overline \nabla \theta(z,y,\mu^{2\alpha j} s + t) \rvert^2 \, dz \, dy \,. \,  \label{eq:rewritingcurlyb}
\end{align}
Using \eqref{eq:boundflow}, we bound the effect of the particle flow for every center $(x,t) \in Q_{\mu^j}^+$ and for every time $s \in [-(1/2)^{2\alpha}, 0]$ by
\begin{align}\label{eq:geometry}
\lvert \mu^{2\alpha j} s +t \rvert &\leq 2 \mu^{2\alpha j} \leq (4 \mu^j)^{2\alpha}\,,  \\
 \lvert R_j(s) + x \rvert &\leq  \mu^{j} + 2 \mu^{(2\alpha -2/q)j} \norm{u}_{L^\infty([t- \mu^{2\alpha j}, t], L^q(\R^2))} \nonumber\\
 &= (4\mu^j)^{2\alpha - 2/q} \Big( \frac{1}{4} (4 \mu^{j})^{1-2\alpha+2/q} + \frac{1}{4^{2\alpha -2/q}} 2 \norm{u}_{L^\infty([- (4\mu^j)^{2\alpha},(4\mu^j)^{2\alpha}], L^q(\R^2))} \Big) \\
 &\leq \frac{5}{8} (4\mu^j)^{2\alpha - 2/q} K_{q}(u; 4\mu^j)\,,
 \label{eq:geometry2}
 \end{align}
where the last inequality follows from the parameter inequality \eqref{eq:stupidparameters}. From \eqref{eq:geometry2} we deduce the set inclusion 
\begin{align}
B_\frac{\mu^j}{2}(R_j(s)+ x)
 &\subseteq B_{\frac{3}{4}K_q(u; 4\mu^j) (4\mu^j)^{2\alpha-2/q}}\,.
\end{align}
The validity of \eqref{eq:boundcurlyb} now follows from \eqref{eq:rewritingcurlyb} after a change of variables in time. The estimate \eqref{eq:boundcurlyt} follows analogously.
\end{proof}

\begin{lemma}\label{lem:limsup} Let $\alpha \in (\frac{1}{4}, \frac{1}{2})$, $p> \frac{1+\alpha}{\alpha}\,,$ $\sigma \in (0, 2\alpha)$ and $q \geq 10\,.$ Assume that 
\begin{equation}\label{e:parameter}
(\sigma - (1-2\alpha)) \frac{p}{1+\alpha} - 2\alpha >0 \,.
\end{equation}
Then, there exist a universal $C= C(\alpha, p, \sigma) \geq 1$ such that 
\begin{equation}
\limsup_{r \to 0} \mathrm{T}(\theta; x,t,r) \leq  C \limsup_{r \to 0 } \mathrm{B}(\theta; x,t,r) \,.
\end{equation}
\end{lemma}
\begin{proof} The lemma is a straightforward consequence of the following claim: if \eqref{e:parameter} holds, then for every $ \rho \in (0, \frac{1}{4})$ small enough, there exists $C=C(\rho, p, \sigma, \alpha) \geq 1$ such that 
\begin{equation}\label{eq:reduction}
\mathrm{T}(\theta; x,t,\rho r) \leq \frac{1}{2} \,\mathrm{T}(\theta; x,t,r) + C\, \mathrm{B}(\theta; x,t,r) \,.
\end{equation}
By translation and scaling invariance, it is enough to prove \eqref{eq:reduction} for $(x,t)=(0,0)$ and $r=1$.  We estimate for fixed time $s$ and fixed center $z$ by the triangular inequality and the Poincar\'e inequality  \eqref{eq:Poincare} 
\begin{align}
\Big( \mean{B_R(z)} \lvert \theta(z', s) - [\theta(s)]_{B_\frac{\rho}{4}(z) }\rvert^\frac{3}{2} \, dz' \Big)^\frac{2p}{3(1+\alpha)} &\leq 2^{\frac{p}{1+\alpha}-1} \Big( \mean{B_R(z)} \lvert \theta(z',s)- [\theta(s)]_{B_\frac{1}{4}(z) }\rvert^\frac{3}{2} \, dz' \Big)^\frac{2p}{3(1+\alpha)} \nonumber \\
&+ C \rho^{-\frac{4p}{3(1+\alpha)}} \norm{ \theta(s)}_{L^\infty(\R^2)}^{\frac{p}{1+\alpha}-2} \int_{B_{1/2}^*(z)} y^b \lvert \overline \nabla \theta^* \rvert^2 \, dz' \, dy\,. \label{eq:noname}
\end{align}
If  $\frac{\rho}{4} \leq R \leq \frac{1}{4}\,,$  also the first term on the right-hand side can be estimated by the term \eqref{eq:noname}. Observe that if $\lvert z \rvert \leq \frac{3}{4} K_q(u; \rho) \rho^{2\alpha - 2/q}\,,$ it holds in particular that 
\begin{equation}
\lvert z \rvert + \frac{1}{2} \leq \frac{3}{4} K_q(u; \rho) + \frac{1}{4}K_q(u; 1) \leq K_q(u;1)
\end{equation}
by the definition of $K_q$ in \eqref{def:Kq}. In other words, $B_{1/2}(z) \subseteq B_{K_q(u;1)}\,.$ Splitting thus the supremum over $ \{ R \geq \frac{1}{4} \}$ and $\{ \frac{\rho}{4} \leq R \leq \frac{1}{4} \}$, this shows that 
\begin{equation}
\mathrm{T}(\theta; \rho) \leq 2^{\frac{p}{1+\alpha}-1} \rho^{(\sigma-(1-2\alpha)) \frac{p}{1+\alpha}- 2\alpha} \,  \mathrm{T}(\theta; 1) + C  \rho^{-\frac{4p}{3(1+\alpha)}} \, \mathrm{B}(\theta; 1) 
\end{equation}
with $C=C(\alpha, p)\geq1\,.$ Thus choosing $\rho$ sufficiently small, we reach \eqref{eq:reduction}.
\end{proof}

\begin{proof}[Proof of Theorem~\ref{thm:running}] By translation invariance, we may assume that $(x,t)=(0,0)\,.$ Let $\alpha \in (\frac{1}{\sqrt{6}}, \frac{1}{2})\,,$ $q\geq 20$ and set $p:= \frac{1+\alpha}{\alpha} + \frac{1}{q}$ as in the statement. Furthermore, we set 
\begin{equation}\label{def:parameters}
\sigma:= 2\alpha - \frac{1}{q} \quad \text{ and } \quad  \gamma:= \sigma - \frac{2\alpha^2}{1+\alpha} \,.
\end{equation}
This is an admissible choice of parameters for Proposition \ref{prop:ex-dec} and we denote by $\delta'$ the smallness requirement and by $\mu$ the rescaling parameter coming from the proposition for this choice of parameters. Let $(\theta, u)$ be a suitable Leray-Hopf solution of \eqref{eq:SQG}--\eqref{eq:u} such that 
\begin{equation}\label{e:assum}
\limsup_{r \to 0} \mathrm{B}(\theta;r) < \delta
\end{equation}
for some $\delta \in (0, \delta')$ yet to be chosen even smaller. From \eqref{e:assum} and Lemma \ref{lem:limsup}, we deduce that there exists an $r_0 \in (0,1)$, depending on everything, such that $ \mathrm{B}(\theta;r) + \mathrm{ T}(\theta; r)  < 2 C \delta$ for all $r \in (0, r_0]\,$ where $C=C(\alpha, p, \sigma)$ is given by the lemma. Hence, up to considering the rescaled solution $\theta_{r_0/4}$ instead of $\theta\,,$ we may assume without loss of generality that 
\begin{equation}\label{e:assumbetter}
\mathrm{B}(\theta;r) + \mathrm{ T}(\theta; r)  < 2 C \delta  \quad \text{ for all } 0< r \leq 4 \,
\end{equation}
and that $\theta$ is a suitable Leray-Hopf solution on $\R^2 \times [-4^{2\alpha}, 4^{2\alpha}]\,.$ In view of Lemma \ref{lem:boundcurlybt}, we can therefore ensure, with an appropriate choice of $\delta= \delta(\alpha, p, \sigma)$, that
\begin{equation}\label{e:finalclaim}
\sup_{ (x,t) \in Q_{\mu^{j_0}}^+} \left\{ \mathcal{B}(\theta_j^{x,t, \mu}; 1/2) + \mathcal{T}_{\frac{p}{1+\alpha}} (\theta_j^{x,t, \mu};1/2) \right\} < \delta' \quad \text{ for all } j_0 \geq 0 \,.
\end{equation}
In particular, \eqref{hyp:ex-dec} holds and by Proposition \ref{prop:ex-dec}, $\theta$ is $C^\beta$-H\"older continuous in space, uniformly in time, a neighbourhood of $(0,0)$ in spacetime, where
$$\beta:= \gamma - \frac{1}{p-1} \Big[ 1-2\alpha +\frac{2\alpha}{p} \Big]  >1-2\alpha\,.$$
In particular, $(0,0)$ is a regular point. 
\end{proof}

\section{Covering argument: Proof of Theorem~\ref{thm:main}}\label{s:covering}

\begin{proof}[Proof of Theorem \ref{thm:main}] We show that for every $t>0$ and every $q\geq 20$, we have
\begin{equation}\label{finalgoal}
\mathcal{H}^{\beta} \big({\rm Sing} \, \theta \cap (\R^2 \times [t, \infty))\big)<\infty\,,
\end{equation}
where 
\begin{equation}
\beta= \beta(q):= \frac{1}{2\alpha-2/q} \Big[\frac{1}{\alpha} + \frac{1+\alpha}{q}(1-2\alpha) \Big]\,.
\end{equation}
The theorem follows by letting $q \to \infty$ and by writing ${\rm Sing } \, \theta = \bigcup_{n \in \N} {\rm Sing } \, \theta \cap (\R^2 \times [1/n, + \infty))\,.$

For the rest of the proof, we fix $q \geq 20$ and $t>0$, and we set $\mathcal{S}:={\rm Sing } \, \theta \cap (\R^2 \times [t, \infty))\,.$ Theorem \ref{thm:running} and Remark~\ref{rmk:strangegeometry}  give a necessary condition in order to have  $(x,s) \in \mathcal{S}\,,$ namely 
\begin{equation}\label{eq:singsetcond}
 \frac{1}{r^{(2\alpha - 2/q)\beta}}  \int_0^r \int_{B_{ r^{2\alpha - 2/q}}(x, s)} y^b \lvert \overline \nabla \theta^* \rvert^2 \, dz \, d\tau \, d y > \frac{\delta}{2^\beta M}:=\delta_{0}\,,
\end{equation}
for a sequence of radii $r \to 0$ where $M=M(\alpha, q, \norm{\theta_0}_{L^2}, t)\geq 1$.
For every $\varepsilon>0$ fixed, by Vitali's covering Lemma, there exists a countable family of disjoint balls $\{ B^i=B_{r_i^{2\alpha-2/q}}(x_i, s_i) \}_{i \in I}$ in spacetime such that $\diam B^i \leq \varepsilon$,  \eqref{eq:singsetcond} holds and yet 
$\mathcal{S} \subseteq \bigcup_{i \in I}  5 B^i\,.$ Hence, we have
\begin{align}
\mathcal{H}_{10 \, \varepsilon}^\beta (\mathcal{S})  \leq \sum_{i \in I} (10 \, r_i^{2\alpha-2/q})^\beta &\leq 10^\beta \delta_0^{-1} \sum_{i \in I} \int_0^{r_i} \int_{B_{r_i^{2\alpha-2/q}}(x_i, s_i)} y^b \lvert \overline{\nabla} \theta^* \rvert^2 \, dz \, d \tau \, dy \\
&\leq 10^\beta \delta_0^{-1}  \int_0^\infty \int_{\R^2\times (0, \infty)} y^b \lvert \overline{\nabla} \theta^* \rvert^2 \, dz \, d\tau \, dy\,,
\end{align}

By the global control $\theta \in L^2((0, \infty),W^{\alpha,2}(\R^3))$ and Theorem \ref{thm:CF}, we conclude \eqref{finalgoal}.
\end{proof}

\subsection*{Acknowledgements} The authors have been supported by the SNF Grant 182565 ``Regularity issues for the Navier-
Stokes equations and for other variational problems". 

\bibliographystyle{plain}
\bibliography{SQG}

\begin{thebibliography}{10}

\bibitem{BuckmasterShkollerVicol2019}
Tristan Buckmaster, Steve Shkoller, and Vlad Vicol.
\newblock Nonuniqueness of weak solutions to the {SQG} equation.
\newblock {\em Comm. Pure Appl. Math.}, 72(9):1809--1874, 2019.

\bibitem{CaffarelliKohnNirenberg1982}
L.~Caffarelli, R.~Kohn, and L.~Nirenberg.
\newblock Partial regularity of suitable weak solutions of the
  {N}avier-{S}tokes equations.
\newblock {\em Comm. Pure Appl. Math.}, 35(6):771--831, 1982.

\bibitem{CaffarelliSilvestre2007}
Luis Caffarelli and Luis Silvestre.
\newblock An extension problem related to the fractional {L}aplacian.
\newblock {\em Comm. Partial Differential Equations}, 32(7-9):1245--1260, 2007.

\bibitem{CaffarelliVasseur2010}
Luis~A. Caffarelli and Alexis Vasseur.
\newblock Drift diffusion equations with fractional diffusion and the
  quasi-geostrophic equation.
\newblock {\em Ann. of Math. (2)}, 171(3):1903--1930, 2010.

\bibitem{CDLM17}
M.~Colombo, C.~De~Lellis, and A.~Massaccesi.
\newblock The generalized {C}affarelli-{K}ohn-{N}irenberg theorem for the
  hyperdissipative {N}avier-{S}tokes system.
\newblock {\em Comm. Pure Appl. Math.}, 73(3):609--663, 2020.

\bibitem{CH2021}
Maria Colombo and Silja Haffter.
\newblock Estimate on the {D}imension of the {S}ingular {S}et of the
  {S}upercritical {S}urface {Q}uasigeostrophic {E}quation.
\newblock {\em Ann. PDE}, 7(1):Paper No. 6, 2021.

\bibitem{ConstantinCotiZelatiVicol2016}
Peter Constantin, Michele Coti~Zelati, and Vlad Vicol.
\newblock Uniformly attracting limit sets for the critically dissipative {SQG}
  equation.
\newblock {\em Nonlinearity}, 29(2):298--318, 2016.

\bibitem{ConstantinTarfuleaVicol2015}
Peter Constantin, Andrei Tarfulea, and Vlad Vicol.
\newblock Long time dynamics of forced critical {SQG}.
\newblock {\em Comm. Math. Phys.}, 335(1):93--141, 2015.

\bibitem{ConstantinVicol2012}
Peter Constantin and Vlad Vicol.
\newblock Nonlinear maximum principles for dissipative linear nonlocal
  operators and applications.
\newblock {\em Geom. Funct. Anal.}, 22(5):1289--1321, 2012.

\bibitem{ConstantinWu2008}
Peter Constantin and Jiahong Wu.
\newblock Regularity of {H}\"{o}lder continuous solutions of the supercritical
  quasi-geostrophic equation.
\newblock {\em Ann. Inst. H. Poincar\'{e} Anal. Non Lin\'{e}aire},
  25(6):1103--1110, 2008.

\bibitem{ConstantinWu2009}
Peter Constantin and Jiahong Wu.
\newblock H\"{o}lder continuity of solutions of supercritical dissipative
  hydrodynamic transport equations.
\newblock {\em Ann. Inst. H. Poincar\'{e} Anal. Non Lin\'{e}aire},
  26(1):159--180, 2009.

\bibitem{Giusti2003}
Enrico Giusti.
\newblock {\em Direct methods in the calculus of variations}.
\newblock World Scientific Publishing Co., Inc., River Edge, NJ, 2003.

\bibitem{Thesis}
Silja Haffter.
\newblock {\em Blow-up, partial regularity and turbulence in incompressible
  fluid dynamics}.
\newblock Doctoral Thesis, EPFL, Lausanne, 2022.

\bibitem{MR4030399}
Cheng He, Yanqing Wang, and Daoguo Zhou.
\newblock New {$\varepsilon$}-regularity criteria of suitable weak solutions of
  the 3{D} {N}avier-{S}tokes equations at one scale.
\newblock {\em J. Nonlinear Sci.}, 29(6):2681--2698, 2019.

\bibitem{KiselevNazarovVolberg2007}
A.~Kiselev, F.~Nazarov, and A.~Volberg.
\newblock Global well-posedness for the critical 2{D} dissipative
  quasi-geostrophic equation.
\newblock {\em Invent. Math.}, 167(3):445--453, 2007.

\bibitem{KohYang16}
Youngwoo Koh and Minsuk Yang.
\newblock The {M}inkowski dimension of interior singular points in the
  incompressible {N}avier-{S}tokes equations.
\newblock {\em J. Differential Equations}, 261(6):3137--3148, 2016.

\bibitem{Kukavica09}
I.~Kukavica.
\newblock The fractal dimension of the singular set for solutions of the
  {N}avier-{S}tokes system.
\newblock {\em Nonlinearity}, 22(12):2889, 2009.

\bibitem{KukavicaPei12}
Igor Kukavica and Yuan Pei.
\newblock An estimate on the parabolic fractal dimension of the singular set
  for solutions of the {N}avier-{S}tokes system.
\newblock {\em Nonlinearity}, 25(9):2775--2783, 2012.

\bibitem{Leray34}
J.~Leray.
\newblock Sur le mouvement d'un liquide visqueux emplissant l'espace.
\newblock {\em Acta Math.}, 63(1):193--248, 1934.

\bibitem{RobinsonSadowski07}
J.C. Robinson and W.~Sadowski.
\newblock Decay of weak solutions and the singular set of the three-dimensional
  {N}avier-{S}tokes equations.
\newblock {\em Nonlinearity}, 20(5):1185, 2007.

\bibitem{Schlag1996}
Wilhelm Schlag.
\newblock Schauder and {$L^p$} estimates for parabolic systems via {C}ampanato
  spaces.
\newblock {\em Comm. Partial Differential Equations}, 21(7-8):1141--1175, 1996.

\bibitem{MR4030593}
Yanqing Wang and Minsuk Yang.
\newblock Improved bounds for box dimensions of potential singular points to
  the {N}avier-{S}tokes equations.
\newblock {\em Nonlinearity}, 32(12):4817--4833, 2019.

\end{thebibliography}
\end{document}